\documentclass[12pt]{article}
\input epsf.tex
\usepackage{amsthm, amssymb, amsmath}
\usepackage{lineno}
\newtheorem{Theorem}{Theorem}
\newtheorem{lemma}{Lemma}

\newtheorem{corollary}{Corollary}
\newtheorem{example}{Example}
\newtheorem{remark}{Remark}

\begin{document}
\title{\bf{Destabilization, stabilization, and multiple attractors in saturated mixotrophic environments}}
\author{Torsten Lindstr\"{o}m \\
Department of Mathematics \\
Linn{\ae}us University \\
SE-35195 V\"{a}xj\"{o}, SWEDEN\\
\\
Yuanji Cheng\\
School of Technology\\
Malm\"{o} University\\
SE-20506 Malm\"{o}, SWEDEN\\
\\
Subhendu Chakraborty\\
Centre for Ocean Life\\
Technical University of Denmark\\
Kemitorvet\\
DK-2800 Kgs. Lyngby, Denmark}
\date{}
\maketitle

\begin{abstract}
The ability of mixotrophs to combine phototrophy and phagotrophy is now well recognized and found to have important implications for ecosystem dynamics. In this paper we examine the dynamical consequences of the invasion of mixotrophs in a model that is a limiting case of the chemostat. The model is a hybrid of a competition model describing the competition between populations of autotroph and mixotroph for limiting resources, and a predator-prey type model describing the interaction between populations of autotroph and herbivore. Our results show that mixotrophs are able to invade in both autotrophic environments and environments described by interactions between autotrophs and herbivores. The interaction between autotrophs and herbivores might be in equilibrium or cycle. We find that invading mixotrophs have the ability to both stabilize and destabilize autotroph-herbivore dynamics depending on the competitive ability of mixotrophs. Moreover the invasion of mixotrophs can also result in multiple attractors. Therefore, our results reveal important consequences of mixotrophic invasions in ecosystems depending on environmental conditions.
\end{abstract}

\section{Introduction}

A large number of plankton taxa, known as mixotrophs, are able to simultaneously exploit both phototrophic and phagotrophic pathway of nutrition. These mixotrophs play important roles in exporting organic matters to higher trophic levels (Ward and Follows (2016)\nocite{Ward.PNAS:113}), carbon export to deep water (Mitra et al. (2014)\nocite{Mitra.PoO:129}), nutrient cycling (Stoecker et al. (2017)\nocite{Stoecker.ARMS:9}) and primary production (Chakraborty et al. (2019)\nocite{chak_prep}).  Mixotrophs are found almost everywhere in the illuminated water column, both in freshwater and marine environments (Hartmann et al. (2012\nocite{Hartmann.PNAS:109}); Stoecker et al. (2009\nocite{Stoecker:AME:57})), oligotrophic and eutrophic systems (Burkholder et al., (2008\nocite{Burkholder.HA:8})), and from polar to equatorial regions (Zubkov and Tarran (2008\nocite{Zubkov.Nature:455}); Stoecker et al. (2009\nocite{Stoecker:AME:57}); Sanders and Gast (2012\nocite{Sanders.FMEcol:82})). They are reported in all planktonic functional groups (except diatoms) and differ largely in their mixotrophic types depending on the availability of light, nutrient and/or prey or other particles upon which they feed (Flynn et a. (2013)\nocite{Flynn.JPR:35}; Chakraborty et al. (2017)\nocite{Chakraborty.AmNat:189}; Berge et al. (2017)\nocite{Berge:ISME:11}). How different types of mixotrophs with different competitive abilities affect the ecosystem dynamics is therefore an interesting question in ecology.

In most of the mathematical models of plankton community, plankton are mainly divided into phototrophs and heterotrophs. Among limited studies that include mixotrophs, only few of them investigated the impact of mixotrophs on the system dynamics. Jost et al. (2004)\nocite{Jost.TPB:66} considered a simple food web model with four variables (nutrient, autotrophs, herbivores and mixotrophs) where mixotrophs can consume nutrients as well as graze on autotrophs, and investigated the importance of inclusion of the mixotrophic link in microbial food webs. By incorporating different types of mixotrophs in a phytoplankton-zooplankton system, Hammer and Pitchford (2005)\nocite{Hammer.ICESMarine:62} examined the consequences of incorporation of mixotrophy on the system's equilibrium structure, stability, short-term dynamics and productivity. Similarly, various physiological types of mixotrophs are incorporated separately in a nutrient-phytoplankton-zooplankton-detritus model by Stickney et al. (2000)\nocite{Stickney.EcMod:125} to observe the effects of different types of mixotrophy on the trophic dynamics of ecosystems. Incorporating mixotrophs in a nutrient-phytoplankton-zooplankton-bacteria system with cell quotas and two essential nutrients, carbon and phosphorous, Crane and Grover (2010)\nocite{Crane.JTB:262} discussed the role of the degree of mixotrophy on the system dynamics under different environmental conditions.

The present study goes further than its predecessors in several aspects. Similar to Jost et al. (2004)\nocite{Jost.TPB:66}, here we consider a chemostat model consisting of nutrient-autotroph-herbivore-mixotroph with more realistic resource and prey uptake term for mixotroph representing the shift in uptake preference depending on the resource and prey availability. Moreover, instead of doing logistic approximations that have already received much criticism (Kooi et al. (1998)\nocite{Kooi.BoMB:60}), we analyze a limiting case of the chemostat that will still provide complete information of the system ({Lindstr\"{o}m and Cheng (2015)})\nocite{lindstr_cheng} similar to the logistic approximations (Kuang and Freedman (1988))\nocite{kuangfyra}. We choose a spectrum of mixotroph types by varying the competitive ability of mixotrophs compared to pure phototrophs and pure herbivores. Similar to the previous studies, here also we assume that mixotrophs are less competitive than pure phototrophs for inorganic resources (Litchman et al. (2007)\nocite{Litchman.EcolLett:10}) and less competitive than pure herbivores for prey (Zubkov and Tarran (2008)\nocite{Zubkov.Nature:455}). This assumption resembles real observations that mixotrophs obtain benefit in competition over specialists mainly under limitation by multiple resources that favors generalists (Rothhaupt, (1996)\nocite{Rothhaupt.Ecology:77}; Katechakis and Stibor, (2006)\nocite{Katechakis.Oecologia:148}). The main aim of the present study is to investigate the dynamical consequences of mixotrophic invasion, especially how the invasion of mixotrophs affects the stability of the system. We perform a rigorous analytical study of the system and support our conclusions with numerical simulations.

Our paper is organized as follows. The model is introduced in Section \ref{modelsection} together with its basic properties and parameter restrictions. In Section \ref{saturated_case} we justify our selection of bifurcation parameters. In Section \ref{selection_of_properties} we derive some properties of the bifurcation diagram that will help validating the numerical part of this paper. In Section \ref{number_eq} the number of equilibria in the model is analyzed and in Section \ref{stab_bound_eq}-\ref{stab_int_eq} the stability of the equilibria is analyzed. The analytical study is corroborated and accomplished with a numerical study in Section \ref{numres}. We end up with a discussion of the consequences of our work in Section \ref{dicussion}.

\section{The model}
\label{modelsection}

In this paper we study an important special case of the following chemostat based model for mixotrophs
\begin{eqnarray}
\frac{dS}{dT}&=&CD-DS-\frac{A_1SX}{1+A_1B_1S}-\frac{A_2SZ}{1+A_2B_2S+A_4B_4X},\nonumber\\
\frac{dX}{dT}&=&\frac{M_1A_1SX}{1+A_1B_1S}-DX-\frac{A_3XY}{1+A_3B_3X}-\frac{A_4XZ}{1+A_2B_2S+A_4B_4X},\nonumber\\
\frac{dY}{dT}&=&\frac{M_3A_3XY}{1+A_3B_3X}-DY,\label{chemostat_mixo_orig_alt2}\\
\frac{dZ}{dT}&=&\frac{(M_2A_2S+M_4A_4X)Z}{1+A_2B_2S+A_4B_4X}-DZ.\nonumber
\end{eqnarray}
Here, $S$ stands for nutrient concentration, $X$ for concentration of autotrophs, $Y$ for concentration of herbivores, and $Z$ for concentration of mixotrophs.
The parameter $C$ represents the nutrient input concentration, $D$ is the dilution rate, $A_i$'s are the search rates, $B_i$'s are the handling times, and $M_i$'s are the conversion factors. A division of the time-budget for the mixotrophs' searching for nutrient and autotrophs is assumed, but no preference between different recourses is considered as in the case of optimal foraging theory models, see e. g. Krebs and Davies (1993)\nocite{Krebs.Introduction}, Lindstr\"{o}m (1994)\nocite{MinPhD}, or Boukal and Krivan (1999)\nocite{Boukal.JoMB:39}. When the Holling (1959)\nocite{Holling.CanEnt:91} arguments are repeated for this situation, the amount of nutrient ${\cal{S}}$ and autotrophic organism ${\cal{X}}$ consumed during the food-gathering period ${\cal{T}}$ is given by
\begin{eqnarray*}
{\cal{S}}&=&A_2S({\cal{T}}-B_2{\cal{S}}-B_4{\cal{X}}),\\
{\cal{X}}&=&A_4X({\cal{T}}-B_2{\cal{S}}-B_4{\cal{X}}).
\end{eqnarray*}
Using Cramer's rule we can directly find a solution of this linear system in the form
\begin{eqnarray*}
\frac{{\cal{S}}}{{\cal{T}}}&=&\frac{A_2S}{1+A_2B_2S+A_4B_4X},\\
\frac{{\cal{X}}}{{\cal{T}}}&=&\frac{A_4X}{1+A_2B_2S+A_4B_4X}.
\end{eqnarray*}
corresponding to the functional responses used in (\ref{chemostat_mixo_orig_alt2}).

In this paper we are particularly interested in the special case $M_1=M_2=M_3=M_4=1$ since it allows for a more complete analysis than other cases. Indeed, consider the functional $H(S,X,Y,Z)=S+X+Y+Z-C$. We get
\begin{displaymath}
\frac{dH}{dT}=CD-DS-DX-DY-DZ=-DH.
\end{displaymath}
That is, $H\rightarrow 0$ along the solution curves of (\ref{chemostat_mixo_orig_alt2}). Therefore, we can hope that removal of one of the variables can provide some asymptotic information regarding the system. By substituting $S=C-X-Y-Z$ in (\ref{chemostat_mixo_orig_alt2}), we get
\begin{eqnarray}
\frac{dX}{dT}&=&\frac{A_1(C-X-Y-Z)X}{1+A_1B_1(C-X-Y-Z)}-DX-\frac{A_3XY}{1+A_3B_3X}\nonumber\\
&&-\frac{A_4XZ}{1+A_2B_2(C-X-Y-Z)+A_4B_4X},\nonumber\\
\frac{dY}{dT}&=&\frac{A_3XY}{1+A_3B_3X}-DY,\label{chemostat_mixo_asymp_altII}\\
\frac{dZ}{dT}&=&\frac{A_2(C-X-Y-Z)Z+A_4XZ}{1+A_2B_2(C-X-Y-Z)+A_4B_4X}-DZ.\nonumber
\end{eqnarray}
The last argument is not always valid, see e. g. Thieme (1992)\nocite{JoMB.Thieme:30}. The asymptotic information provided by the reduced system like (\ref{chemostat_mixo_asymp_altII}) will not always reflect all details of the limit sets of the original system (\ref{chemostat_mixo_orig_alt2}). Hence, there are questions at several instances related to the selection of our special case with no natural death rates and all conversion factors equal to one. Our hope is anyway that the special cases selected for detailed analytical study here contain substantial information of the properties of nearby more realistic models.

In order to be a valid model, (\ref{chemostat_mixo_asymp_altII}) needs to satisfy some conditions. We express these conditions as follows:
\begin{itemize}
\item[(A)] The autotrophs should be better competitors for nutrient than the mixotrophs. A criterion for this can be formulated by considering the dynamical properties of the system (\ref{chemostat_mixo_asymp_altII}) when $Y=0$ and $A_4=0$ or
\begin{eqnarray}
\frac{dX}{dT}&=&\frac{A_1(C-X-Z)X}{1+A_1B_1(C-X-Z)}-DX,\nonumber\\
\frac{dZ}{dT}&=&\frac{A_2(C-X-Z)Z}{1+A_2B_2(C-X-Z)}-DZ.\label{competitive_system_auto_mixo}
\end{eqnarray}
We end up (see Lemma \ref{lemma1} below) with a condition that involves the fixed points of the systems
\begin{displaymath}
\frac{dX}{dT}=\frac{A_1(C-X)X}{1+A_1B_1(C-X)}-DX
\end{displaymath}
and
\begin{displaymath}
\frac{dZ}{dT}=\frac{A_2(C-Z)Z}{1+A_2B_2(C-Z)}-DZ,
\end{displaymath}
respectively.
This inequality then reads
\begin{equation}
\frac{A_1C(1-DB_1)-D}{A_1(1-DB_1)}>\frac{A_2C(1-DB_2)-D}{A_2(1-DB_2)}.
\label{fixed_pt_larger_cond_orig}
\end{equation}
Standard algebraic manipulation reduces this inequality to
\begin{equation}
A_1-A_2>D\left(A_1B_1-A_2B_2\right).
\label{fixed_pt_larger_cond}
\end{equation}
The consequences of this inequality for the competitive dynamics between the autotroph and the mixotroph are given in Lemma \ref{lemma1} below.
\item[(B)] The herbivore is a better grazer on autotrophs than the mixotroph. Criteria for this can be formulated by considering the dynamical properties of the system (\ref{chemostat_mixo_asymp_altII}) when $A_2=0$ or
\begin{eqnarray}
\frac{dX}{dT}&=&\frac{A_1(C-X-Y-Z)X}{1+A_1B_1(C-X-Y-Z)}-DX-\frac{A_3XY}{1+A_3B_3X}\nonumber\\
&&-\frac{A_4XZ}{1+A_4B_4X},\nonumber\\
\frac{dY}{dT}&=&\frac{A_3XY}{1+A_3B_3X}-DY,\label{competitive_system_herbi_mixo}\\
\frac{dZ}{dT}&=&\frac{A_4XZ}{1+A_4B_4X}-DZ.\nonumber
\end{eqnarray}
In order to make sure that the ratio $Y/Z$ always increases, we have
\begin{displaymath}
\frac{A_3X}{1+A_3B_3X}>\frac{A_4X}{1+A_4B_4X}, \:\: 0< X\leq C
\end{displaymath}
which is equivalent to $A_3-A_4>0$ and
\begin{equation}
A_3-A_4>A_3A_4C(B_3-B_4).
\label{global_condition_B}
\end{equation}
The above condition holds if $A_3>A_4$ (the herbivore searches for autotrophs more efficiently than the mixotroph) and $B_3<B_4$ (the mixotroph needs more handling time for the autotroph than the herbivore). This global condition can certainly be improved, see e. g. Kustarov (1986)\nocite{kustaroveng} and Lindstr\"{o}m (1994, 2000)\nocite{konflul,szeged1}. We shall derive a local condition for this, as well. The above condition ensures that the herbivore is a better grazer on autotroph in equilibrium in the presence of autotroph (without looking whether this equilibrium is stable or not). We note that the last two equations for the subsystem of (\ref{competitive_system_herbi_mixo}) take the form
\begin{eqnarray}
\frac{dY}{dT}&=&\frac{A_3XY}{1+A_3B_3X}-DY=\Psi_1(X)Y,\nonumber\\
\frac{dZ}{dT}&=&\frac{A_4XZ}{1+A_4B_4X}-DZ=\Psi_2(X)Z.\label{chemostat_mixo_asymp_row2_3}
\end{eqnarray}
This means that the last row vector of the Jacobian matrix evaluated at the unique interior fixed point at the $X,Y$-plane takes the form $\left(0,0,\Psi_2(X_\star)\right)$ assuming that $X_\star$ is the positive solution of $\Psi_1(X)=0$. In other words, the eigenvalue corresponding to invasion of the mixotroph near this equilibrium is given by $\Psi_2(X_\star)$. Requiring that the mixotroph cannot invade near this equilibrium is now equivalent to
\begin{equation}
A_3-A_4>D(A_3B_3-A_4B_4)
\label{local_condition_B}
\end{equation}
which is our local condition. Obviously, we have $(\ref{global_condition_B})\Rightarrow (\ref{local_condition_B})$, since (\ref{local_condition_B}) means that
\begin{displaymath}
\Psi_2(X_\star)<0=\Psi_1(X_\star)
\end{displaymath}
whereas (\ref{global_condition_B}) means that
\begin{displaymath}
\Psi_2(X)<\Psi_1(X), \: \forall X\in]0,C].
\end{displaymath}
We experience later that (\ref{local_condition_B}) can be formulated as a restriction of one of the bifurcation parameters of the system, cf. Remark \ref{remark_upper_bound}.
%\item[(C)] The mixotroph concentrates on acting as an autotroph, ie
%\begin{displaymath}
%RA_4<(1-R)A_2,
%\end{displaymath}
%or $R$ is a sufficiently small parameter.
\end{itemize}
Conditions derived in (A) and (B) imply that there is a trade-off between the flexibility of being a mixotroph and specializing as an autotroph or as grazer. Such assumptions have substantial support for many organisms, see Stearns (1992)\nocite{stearns}, Litchman et al. (2007)\nocite{Litchman.EcolLett:10}, and Zubkov and Tarran (2008)\nocite{Zubkov.Nature:455}. We describe the dynamical consequences of (\ref{fixed_pt_larger_cond}) in the following lemma.
\begin{lemma}
If \em (\ref{fixed_pt_larger_cond}), \em then the fixed point
\begin{equation}
\left(\frac{A_1C(1-DB_1)-D}{A_1(1-DB_1)},0\right)
\label{fixedpt_comp_syst_stable}
\end{equation}
attracts solutions of all initial conditions in the triangle $X>0$, $Z\geq 0$, $X+Z\leq C$ for the system \em (\ref{competitive_system_auto_mixo}). \em
\label{lemma1}
\end{lemma}
\begin{proof}
We first prove that the triangle mentioned in the conditions remains invariant. We have solutions at $X=0$, $Z=0$ and by uniqueness of solutions they cannot be intersected. Next consider the functional $H(X,Z)=X+Z$. The total time derivative of this functional with respect to (\ref{competitive_system_auto_mixo}) is given by
\begin{displaymath}
\dot{H}=\frac{A_1(C-X-Z)X}{1+A_1B_1(C-X-Z)}-DX+\frac{A_2(C-X-Z)Z}{1+A_2B_2(C-X-Z)}-DZ.
\end{displaymath}
This quantity is negative at the line $X+Z=C$, for $X>0$ and $Z>0$. The Jacobian of this system evaluated at a generic point of the system (\ref{competitive_system_auto_mixo}) is given by
\begin{displaymath}
J(X,Z)=\left(\begin{array}{cc}
J_{11}(X,Z)&\frac{-A_1X}{(1+A_1B_1(C-X-Z))^2}\\
\frac{-A_2Z}{(1+A_2B_2(C-X-Z))^2}& J_{22}(X,Z)
\end{array}\right).
\end{displaymath}
%with
%\begin{displaymath}
%J_{11}(X,Y)=\frac{-A_1X+A_1(C-X-Z)(1+A_1B_1(C-X-Z))}{(1+A_1B_1(C-X-Z))^2}-D
%\end{displaymath}
%and
%\begin{displaymath}
%J_{22}(X,Y)=\frac{-(1-R)A_2Z+(1-R)A_2(C-X-Z)(1+(1-R)A_2B_2(C-X-Z))}{(1+(1-R)A_2B_2(C-X-Z))^2}-D
%\end{displaymath}
The off-diagonal elements are clearly negative and thus, (\ref{competitive_system_auto_mixo}) is competitive. If a two dimensional system is competitive, then all its solutions converge towards a fixed point, see Smith (1995)\nocite{Smith.monotone}. The system has one equilibrium at (0,0) and it is unstable exactly when the equilibria at (\ref{fixedpt_comp_syst_stable}) and
\begin{displaymath}
\left(0,\frac{A_2C(1-DB_2)-D}{A_2(1-DB_2)}\right)
\end{displaymath}
have one positive co-ordinate, each. By (\ref{fixed_pt_larger_cond}), there are no other equilibria. A sketch of the phase-portrait (see Figure \ref{kurva}) or a check of the eigenvalues shows that (\ref{fixedpt_comp_syst_stable}) is locally stable. Hence, it attracts all initial conditions in the positive quadrant.
\end{proof}
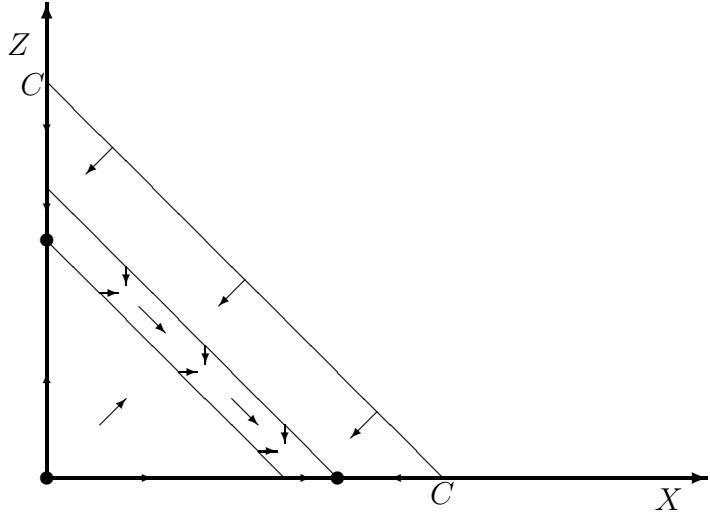
\begin{figure}
\epsfxsize=138mm
\begin{picture}(232,232)(0,0)
\thicklines
\put(60,10){\vector(1,0){250}}
\put(60,10){\vector(0,1){180}}
\thinlines
\put(210,10){\line(-1,1){150}}
\put(170,10){\line(-1,1){110}}
\put(150,10){\line(-1,1){90}}
\put(170,10){\circle*{5}}
\put(60,100){\circle*{5}}
\put(60,10){\circle*{5}}
\put(60,40){\vector(0,1){10}}
\put(60,120){\vector(0,-1){10}}
\put(60,150){\vector(0,-1){10}}
\put(80,30){\vector(1,1){10}}
\put(90,10){\vector(1,0){10}}
\put(150,10){\vector(1,0){10}}
\put(200,10){\vector(-1,0){10}}
\put(140,20){\vector(1,0){7}}
\put(110,50){\vector(1,0){7}}
\put(80,80){\vector(1,0){7}}
\put(150,30){\vector(0,-1){7}}
\put(120,60){\vector(0,-1){7}}
\put(90,90){\vector(0,-1){7}}
\put(130,40){\vector(1,-1){10}}
\put(95,75){\vector(1,-1){10}}
\put(185,35){\vector(-1,-1){10}}
\put(135,85){\vector(-1,-1){10}}
\put(85,135){\vector(-1,-1){10}}
\put(205,0){$C$}
\put(50,155){$C$}
\put(290,-2){$X$}
\put(45,170){$Z$}
\end{picture}
\caption{System (\protect\ref{competitive_system_auto_mixo}) with its three equilibria in the invariant triangle.}
\label{kurva}
\end{figure}

The interpretation of (\ref{fixed_pt_larger_cond}) is thus, that in the absence of grazers, the autotroph is a better competitor for nutrient than the mixotroph. Another interpretation of this result is that we have verified Fisher's (1930)\nocite{fisher} maximum principle for this model.

We introduce dimensionless parameters and variables through the changes
\begin{displaymath}
x=\frac{A_1(1-DB_1)}{A_1C(1-DB_1)-D}X,\:\:
y=\frac{A_1(1-DB_1)}{A_1C(1-DB_1)-D}Y,
\end{displaymath}
\begin{displaymath}
z=\frac{A_1(1-DB_1)}{A_1C(1-DB_1)-D}Z,\:\:
\tau=(A_1C(1-DB_1)-D)T,
\end{displaymath}
\begin{displaymath}
0\leq\kappa_1=\frac{B_1(A_1C(1-DB_1)-D)}{1-DB_1}\leq A_1B_1C=\gamma_1,
\end{displaymath}
\begin{displaymath}
0\leq\kappa_2=\frac{A_2B_2(A_1C(1-DB_1)-D)}{A_1(1-DB_1)}\leq A_2B_2C=\gamma_2,
\end{displaymath}
\begin{displaymath}
0<m=\frac{A_3(1-DB_3)}{A_1(1-DB_1)}\leq\frac{A_3}{A_1(1-DB_1)}=a_1,
\end{displaymath}
\begin{displaymath}
0<a_2=\frac{A_4(1-DB_4)}{A_1(1-DB_1)},
\end{displaymath}
\begin{displaymath}
0\leq b_1=\frac{A_3B_3(A_1C(1-DB_1)-D)}{A_1(1-DB_1)},
\end{displaymath}
\begin{displaymath}
0\leq b_2=\frac{A_4B_4(A_1C(1-DB_1)-D)}{A_1(1-DB_1)},
\end{displaymath}
\begin{displaymath}
0<\frac{A_1D(1-DB_1)}{A_3(1-DB_3)(A_1C(1-DB_1)-D)}=x_\star<1,
\end{displaymath}
\begin{displaymath}
0<k=\frac{A_2(1-DB_2)}{A_1(1-DB_1)}<1,\:{\rm{and}}
\end{displaymath}
\begin{displaymath}
0<c=\frac{A_1(1-DB_1)}{A_2(1-DB_2)}\cdot\frac{A_2C(1-DB_2)-D}{A_1C(1-DB_1)-D}<1.
\end{displaymath}
As a result, system (\ref{chemostat_mixo_asymp_altII}) becomes
\begin{eqnarray}
x^\prime&=&\frac{x(1-x-y-z)}{1+\gamma_1-\kappa_1(x+y+z)}-\frac{a_1xy}{1+b_1x}-\frac{a_2 xz}{1+\gamma_2-\kappa_2(x+y+z)+b_2 x},\nonumber\\
y^\prime&=&\frac{a_1}{1+b_1x_\star}\cdot\frac{x-x_\star}{1+b_1x}y,\label{chemostat_mixo_orig_scaled_altII}\\
z^\prime&=&\frac{k z\left(c-x-y-z\right)+a_2 xz}{1+\gamma_2-\kappa_2(x+y+z)+b_2 x}.\nonumber
\end{eqnarray}
There exists an explicit relationship between $a_1$ and $m$ and it is used in (\ref{chemostat_mixo_orig_scaled_altII}) above for removing the parameter $m$. We also have
\begin{displaymath}
\frac{\gamma_2}{\kappa_2}=\frac{\gamma_1}{\kappa_1}
\end{displaymath}
but prefer to use the expression at the right-hand side in the equations (\ref{chemostat_mixo_orig_scaled_altII}) in order to reduce the number of parameters in the involved functions.

The conditions for a couple of the inequalities above are less clear than the others and are stated in a lemma. Its proof follows standard algebraic procedures.
\begin{lemma}
If \em (\ref{fixed_pt_larger_cond}), \em then we have $0<k<1$, and $0<c<1$.
\label{lemma_kc}
\end{lemma}
\begin{proof}
We use (\ref{fixed_pt_larger_cond}) to prove the inequality for $k$. The inequality $c<1$ follow more easily and directly from (\ref{fixed_pt_larger_cond_orig}) than from (\ref{fixed_pt_larger_cond}).
\end{proof}
We end this section by proving the following theorem.
\begin{Theorem}
Solutions of \em (\ref{chemostat_mixo_orig_scaled_altII}) \em remain positive and bounded.
\end{Theorem}
\begin{proof}
The solutions remain positive since we have solutions in the planes $x=0$, $y=0$, and $z=0$. Uniqueness of solutions grants positive solutions. Next consider the Lyapunov functions $H(x,y,z)=x+y+z$ for $x+y+z\geq 1$. We have from (\ref{chemostat_mixo_orig_scaled_altII})
\begin{eqnarray*}
H^\prime&=&\frac{x(1-x-y-z)}{1+\gamma_1-\kappa_1(x+y+z)}-\frac{a_1xy}{1+b_1x}-\frac{a_2 xz}{1+\gamma_2-\kappa_2(x+y+z)+b_2x}\\
       &&+\frac{a_1}{1+b_1x_\star}\cdot\frac{x-x_\star}{1+b_1x}y+\frac{k z\left(c -x-y-z\right)+a_2 xz}{1+\gamma_2-\kappa_2(x+y+z)+b_2x}\\
       &\leq&-\frac{a_1xy}{1+b_1x}+\frac{a_1}{1+b_1x_\star}\cdot\frac{x-x_\star}{1+b_1x}y\leq 0.
\end{eqnarray*}
\end{proof}
Thus, all solutions of interest will eventually be located in the simplex $x\geq 0$, $y\geq 0$, $z\geq 0$, and $x+y+z\leq 1$.

\section{The saturated almost logistic case}
\label{saturated_case}

The main topic of this paper is the saturated almost logistic case $B_1=B_2=0$, i.e. $\gamma_1=\kappa_1=\gamma_2=\kappa_2=0$. As a result, (\ref{chemostat_mixo_orig_scaled_altII}) takes the form
\begin{eqnarray}
x^\prime &=&x(1-x-y-z)-\frac{a_1xy}{1+b_1x}-\frac{a_2 xz}{1+b_2x},\nonumber\\
y^\prime &=&\frac{a_1}{1+b_1x_\star}\cdot\frac{x-x_\star}{1+b_1x}y,\label{chemostat_mixo_sat_scaled_altII}\\
z^\prime &=&\frac{k z}{1+b_2x}\left(c-x-y-z+\frac{a_2}{k} x\right).\nonumber
\end{eqnarray}
Note that (\ref{chemostat_mixo_orig_scaled_altII}) differs from logistic approximations in the terms $-xy$ and $-xz$ in the equation for $x^\prime$ and that these terms makes the model a limiting case of the chemostat possessing explicit resource dynamics.
It corresponds to the limiting cases $\kappa_i=0,\gamma_i=0,\: i=1,2$ studied in Lindstr\"{o}m and Cheng (2015)\nocite{lindstr_cheng} where the complete global information regarding the dynamics of the subsystem describing the situation when $z=0$ was ensured. We start by doing the nonlinear change
\begin{equation}
\frac{d\tau}{dt}=\frac{1+b_2x}{k}
\label{time_trans}
\end{equation}
of the independent variable and rewriting the above systems in the following isocline form
\begin{eqnarray}
\dot{x}&=&f_i(x)F_i(x)-yf_1(x)-zf_2(x),\nonumber\\
\dot{y}&=&y\psi(x),\label{chemostat_mixo_sat_isocline_altII}\\
\dot{z}&=&z(G(x)-y-z).\nonumber
\end{eqnarray}
where
\begin{displaymath}
f_i(x)=\frac{x+\frac{a_ix}{1+b_ix}}{\frac{k}{1+b_2x}},\:\:F_i(x)=\frac{(1-x)(1+b_ix)}{1+a_i+b_ix},\:\: i=1,2
\end{displaymath}
and
\begin{displaymath}
\psi(x)=\frac{\frac{a_1}{1+b_1x_\star}\cdot\frac{x-x_\star}{1+b_1x}}{\frac{k}{1+b_2 x}},\:\: G(x)=c-x+\frac{a_2}{k}x.
\end{displaymath}
We begin with a lemma stating the most important properties of the functions involved here.
\begin{lemma}
We have that $f_i^\prime(x)>0$, $x\geq 0$, $i=1,2$, and $\psi^\prime(x_\star)>0$.
\end{lemma}
\begin{proof}
Since $\psi$ is differentiable and $(x-x_\star)\psi(x)$ is positive definite, $\psi^\prime(x_\star)>0$. For $f_2$ we get
\begin{displaymath}
f_2(x)=\frac{x}{k}(1+a_2+b_2x),\: {\rm{and}}\: f_2^\prime(x)=\frac{1}{k}(1+a_2+2b_2x).
\end{displaymath}
It is obvious that $f_2$ increases for $x\geq 0$. For $f_1$ we have
\begin{displaymath}
f_1(x)=\frac{x(1+a_1+b_1x)(1+b_2x)}{k(1+b_1x)}=\frac{x+a_1x+b_1x^2+b_2x^2+a_1b_2x^2+b_1b_2x^3}{k(1+b_1x)}.
\end{displaymath}
We differentiate and get
\begin{eqnarray*}
k(1+b_1x)^2f_1^\prime(x)&=&(1+a_1+2b_1x+2b_2x+2a_1b_2x+3b_1b_2x^2)(1+b_1x)\\
&&-b_1(x+a_1x+b_1x^2+b_2x^2+a_1b_2x^2+b_1b_2x^3) \\
  &=&1+a_1+2b_1x+2b_2x+2a_1b_2x+3b_1b_2x^2\\
  &&+b_1x+a_1b_1x+2b_1^2x^2+2b_1b_2x^2+2a_1b_1b_2x^2+3b_1^2b_2x^3\\
  &&-b_1x-a_1b_1x-b_1^2x^2-b_1b_2x^2-a_1b_1b_2x^2-b_1^2b_2x^3\\
  &=&1+a_1+2b_1x+2b_2x+2a_1b_2x\\
  &&+b_1^2x^2+4b_1b_2x^2+a_1b_1b_2x^2+2b_1^2b_2x^3>0,\\
\end{eqnarray*}
for $x\geq 0$.
\end{proof}
We continue with an analysis of the various subsystems. If $y=0$ and $z=0$ we get the logistic equation $\dot{x}=x(1-x)$. For $x=0$, $z=0$, we get \begin{displaymath}
\dot{y}=-\frac{\frac{a_1}{k}x_\star y}{1+b_1x_\star}
\end{displaymath} and for $x=0$, $y=0$ we get $\dot{z}=z(c-z)$, respectively. We return to our logistic friend also in the $yz$-plane. Now we assume $z=0$ and consider
\begin{eqnarray}
  \dot{x} &=& f_1(x)\left(F_1(x)-y\right),\nonumber\\
  \dot{y} &=& \psi(x)y.\label{sat-pp}
\end{eqnarray}
This system turns out to be a Gause-type predator-prey system and its qualitative dynamics is completely known, see Lindstr\"{o}m and Cheng (2015). We may, if necessary, consider the system without the transformation of the independent variable (\ref{time_trans}). Indeed, the sign of the derivative of $F$ determines a lot of its dynamical properties and the derivative is given by
\begin{displaymath}
F_1^{\prime}(x)=\frac{-1-a_1+a_1b_1-2b_1(a_1+1)x-b_1^2x^2}{(1+a_1+b_1x)^2}.
\end{displaymath}
Its denominator is always positive, but its nominator has zeros at
\begin{displaymath}
x=\frac{-a_1-1\pm\sqrt{a_1^2+a_1+a_1b_1}}{b_1}.
\end{displaymath}
One of these zeros is always negative and there exits a positive zero if $b_1>1+\frac{1}{a_1}$. The following lemmas follow from Lindstr\"{o}m and Cheng (2015)\nocite{lindstr_cheng}.
\begin{lemma}
Consider \em(\ref{sat-pp}). \em If either $b_1\leq 1+\frac{1}{a_1}$ or
\begin{displaymath}
0<\frac{-a_1-1+\sqrt{a_1^2+a_1+a_1b_1}}{b_1}\leq x_\star
\end{displaymath}
then the equilibrium $(\min({x_\star},1),\max(0,F_1(x_\star)))$ is globally asymptotically stable in the cone $x>0,y>0$.
\label{glob_stab}
\end{lemma}
\begin{lemma}
If
\begin{displaymath}
0<x_\star<\frac{-a_1-1+\sqrt{a_1^2+a_1+a_1b_1}}{b_1}
\end{displaymath}
then \em (\ref{sat-pp}) \em possesses a unique limit cycle that is stable.
\label{unique_limit_cycle}
\end{lemma}

We are now ready with the subsystem at the $z=0$ plane and turn over to the subsystem at the $y=0$ plane. We obtain the competition system
\begin{eqnarray}
\dot{x} &=& f_2(x)\left(F_2(x)-z\right),\nonumber\\
\dot{z}&=&z\left(G(x)-z\right).\label{sat_comp_altII}
\end{eqnarray}
Of course, the properties of $F_i,\: i=1,2$ are identical. The solutions $G(x)-F_2(x)=0$ determine the number of interior fixed points in the $xz$-plane. Since cubic terms do not enter in this equation, there exist at most two interior fixed points in the positive quadrant of this plane and these fixed points can be computed explicitly. We have now identified the following equilibria of (\ref{chemostat_mixo_sat_scaled_altII}): $(0,0,0)$ (washout/extinction), $(1,0,0)$ (carrying capacity/survival), $(0,0,c)$ (mixotroph carrying capacity/mixotroph survival), $(x_\star,F_1(x_\star),0)$ (predator-prey) that always exist. In addition, we have possibly two competition equilibria $(x_\pm,0,F_2(x_\pm))$. The values $x_\pm$ are the solutions of the quadratic equation
\begin{equation}
k c\left(a_2-\frac{1-c}{c}\right)+\left(a_2-(k-1)\right)a_2 x-b_2 k(1-c)x+a_2 b_2 x^2=q(x)=0,
\label{multiple_eq_alt_II}
\end{equation}
in the unit interval. Finally we have possibly one coexistence equilibrium that we denote by $(x_\star,y_\star,z_\star)$. The last two coordinates of the equilibria are given by the solutions of the linear system
\begin{eqnarray*}
x_\star(1-x_\star)&=&f_1(x_\star)y+f_2(x_\star)z, \\
G(x_\star)&=&y+z.
\end{eqnarray*}
Here, Cramer's rule provides the solutions
\begin{eqnarray*}
y_\star&=&\frac{f_2(x_\star)(F_2(x_\star)-G(x_\star))}{f_1(x_\star)-f_2(x_\star)}=\frac{k(1-x_\star)-\left(1+\frac{a_2}{1+b_2x_\star}\right)\left(k(c-x_\star)+a_2x_\star\right)}{\frac{k}{x_\star}\left(\frac{a_1x_\star}{1+b_1x_\star}-\frac{a_2x_\star}{1+b_2x_\star}\right)},
\\
z_\star&=&\frac{f_1(x_\star)(G(x_\star)-F_1(x_\star))}{f_1(x_\star)-f_2(x_\star)}=\frac{\left(1+\frac{a_1}{1+b_1x_\star}\right)\left(k(c-x_\star)+a_2x_\star\right)-k(1-x_\star)}{\frac{k}{x_\star}\left(\frac{a_1x_\star}{1+b_1x_\star}-\frac{a_2x_\star}{1+b_2x_\star}\right)}.
\end{eqnarray*}
We conclude first that the denominators in the expressions for $y_\star$ and $z_\star$ are positive.
\begin{lemma}
Each of the following are equivalent:
\begin{enumerate}
  \item [\em(i) \em] The local condition \em (\ref{local_condition_B}) \em
  \item [\em(ii)\em] $f_1(x_\star)>f_2(x_\star)$.
  \item [\em(iii)\em] $a_1-a_2>(a_2b_1-a_1b_2)x_\star$
\end{enumerate}
\label{lemma_loc_cond}
\end{lemma}
\begin{proof}
We start from $f_1(x_\star)>f_2(x_\star)$, or (ii). After cancelling of some positive factors and terms, this condition is equivalent to
\begin{equation}
\frac{a_1}{1+b_1x_\star}>\frac{a_2}{1+b_2x_\star}
\label{parametric_local_cond_B}
\end{equation}
phrased in our current parameters. Since the denominators are positive, this condition is equivalent to (iii) by standard algebraic procedures. Direct substitution of the original parameters in (\ref{parametric_local_cond_B}) gives now
\begin{displaymath}
\frac{\frac{A_3}{A_1(1-DB_1)}}{1+\frac{B_3D}{(1-DB_3)}}>\frac{\frac{A_4}{A_1(1-DB_1)}}{1+\frac{A_4B_4D}{A_3(1-DB_3)}}.
\end{displaymath}
Standard algebraic procedures reduces this expression to (\ref{local_condition_B}) and these steps can be carried out in the opposite direction, too. Hence (i), (ii), and (iii) are equivalent.
\end{proof}
\begin{remark}\em
We note that (\ref{parametric_local_cond_B}) puts an upper bound for $a_2$ given by
\begin{equation}
a_2<a_1\frac{1+b_2x_\star}{1+b_1x_\star}.
\label{upper_bound_a2}
\end{equation}
Since $a_2$ is going to be a bifurcation parameter in our subsequent numerical study, such bounds are quite useful.
\label{remark_upper_bound}
\em\end{remark}
We conclude that the interior equilibrium exists (phrased in terms of the mixotrophic parameters) if and only if $\breve{a}_2<a_2<\tilde{a}_2$ with
\begin{eqnarray*}
\breve{a}_2&=&\frac{k(1-x_\star)-k(c-x_\star)\left(1+\frac{a_1}{1+b_1x_\star}\right)}{x_\star+\frac{a_1x_\star}{1+b_1x_\star}},\\
\tilde{a}_2&=&\frac{-(1-\frac{k}{1+b_2x_\star})x_\star-\frac{kc}{1+b_2x_\star}+\sqrt{\left(\left(1-\frac{k}{1+b_2x_\star}\right)x_\star+\frac{kc}{1+b_2x_\star}\right)^2+\frac{4x_\star k(1-c)}{1+b_2x_\star}}}{\frac{2x_\star}{1+b_2x_\star}}
\end{eqnarray*}
and that a similar criterion in terms of the involved functions can be stated as
\begin{equation}
F_1(x_\star)<G(x_\star)<F_2(x_\star),
\label{iso_interior_crit}
\end{equation}
see Figure \ref{F2_G_F1}.

\section{A selection of properties of the bifurcation diagram}
\label{selection_of_properties}

The last conclusions in the previous section states that $a_2$ and $x_\star$ are important bifurcation parameters not only in the competition plane and the predator-prey plane, respectively, but also for the entire system (\ref{chemostat_mixo_sat_scaled_altII}). For instance, there are possibilities that an interval with respect to the mixo\-tro\-phic link exists, granting equilibrium coexistence. A number of results that specify whether different regions in the bifurcation diagram exist or not can now be formulated. They will be used later on for selecting parameters for a sufficiently general numerically computed bifurcation diagram and for validating our numerical results.

We begin with formulating a condition for $\breve{a}_2<\tilde{a}_2$ indicating that restrictions must be put on $k$ if $c$ is small in order to ensure an interior equilibrium point.
\begin{lemma} Define
\begin{displaymath}
\hat{k}=\left\{\begin{array}{cc}
                  \infty &,{\rm{if}}\:\: c\geq x_\star+\frac{1-x_\star}{1+\frac{a_1}{1+b_1x_\star}}, \\
                  \frac{x_\star a_1}{\frac{1-x_\star}{1+a_1}+x_\star-c} &,{\rm{if}} \:\: c<x_\star+\frac{1-x_\star}{1+\frac{a_1}{1+b_1x_\star}}.
                \end{array}\right.
\end{displaymath}
We have $\breve{a}_2<\tilde{a}_2\Leftrightarrow 0<k<\min(1,\hat{k})$.
\em
\label{exist-param_sat}
\end{lemma}

\begin{proof} We commence by introducing the auxiliary variable
\begin{equation}
s=\frac{1-x_\star-\left(1+\frac{a_1}{1+b_1x_\star}\right)(c-x_\star)}{(1+b_2x_\star)\left(1+\frac{a_1}{1+b_1x_\star}\right)}
=\frac{1-x_\star}{(1+b_2x_\star)\left(1+\frac{a_1}{1+b_1x_\star}\right)}-\frac{c-x_\star}{1+b_2x_\star}.
\label{s_def_sat}
\end{equation}
The inequality $\breve{a}_2<\tilde{a}_2$ is then equivalent to
\begin{eqnarray}
&&2ks+\left(1-\frac{k}{1+b_2x_\star}\right)x_\star+\frac{kc}{1+b_2x_\star}\nonumber\\&<&\sqrt{\left(\left(1-\frac{k}{1+b_2x_\star}\right)x_\star+\frac{kc}{1+b_2x_\star}\right)^2
+\frac{4x_\star k(1-c)}{1+b_2x_\star}}.\label{s_ineq_sat}
\end{eqnarray}
The radical expression in (\ref{s_ineq_sat}) is positive by Lemma \ref{lemma_kc}. Inequality (\ref{s_ineq_sat}) is therefore equivalent to (\ref{s_ineq2_sat}) after that the squares have been cancelled and the common positive factor $4k$ has been cancelled. Further reorganization of the inequality and substitution of the expression (\ref{s_def_sat}) gives
\begin{eqnarray}
ks^2+s\left(\left(1-\frac{k}{1+b_2x_\star}\right)x_\star+\frac{kc}{1+b_2x_\star}\right)&<&\frac{x_\star (1-c)}{1+b_2x_\star}\label{s_ineq2_sat}\\
&\Updownarrow&\nonumber\\
ks\left(s+\frac{c-x_\star}{1+b_2x_\star}\right)+sx_\star&<&\frac{x_\star (1-c)}{1+b_2x_\star}\nonumber\\
&\Updownarrow&\nonumber\\
ks\frac{1-x_\star}{(1+b_2x_\star)\left(1+\frac{a_1}{1+b_1x_\star}\right)}+&&\nonumber\\
+\left(\frac{1-x_\star}{(1+b_2x_\star)\left(1+\frac{a_1}{1+b_1x_\star}\right)}-\frac{c-x_\star}{1+b_2x_\star}\right)x_\star&<&\frac{x_\star(1-c)}{1+b_2x_\star}\nonumber\\
&\Updownarrow&\nonumber\\
ks\frac{1-x_\star}{1+\frac{a_1}{1+b_1x_\star}}+
\left(\frac{1-x_\star}{1+\frac{a_1}{1+b_1x_\star}}-c+x_\star\right)x_\star&<&x_\star(1-c)\nonumber\\
&\Updownarrow&\nonumber\\
ks\frac{1-x_\star}{1+\frac{a_1}{1+b_1x_\star}}+\frac{1-x_\star}{1+\frac{a_1}{1+b_1x_\star}}x_\star&<&x_\star(1-x_\star)\nonumber\\
&\Updownarrow&\nonumber\\
\frac{ks}{1+\frac{a_1}{1+b_1x_\star}}+\frac{x_\star}{1+\frac{a_1}{1+b_1x_\star}}&<&x_\star\nonumber\\
&\Updownarrow&\nonumber\\
\frac{ks}{1+\frac{a_1}{1+b_1x_\star}}&<&x_\star-\frac{x_\star}{1+\frac{a_1}{1+b_1x_\star}}=\frac{\frac{x_\star a_1}{1+b_1x_\star}}{1+\frac{a_1}{1+b_1x_\star}}\nonumber\\
&\Updownarrow&\nonumber\\
ks&<&\frac{x_\star a_1}{1+b_1x_\star}\label{last_s_ineq_sat}
\end{eqnarray}
The required restriction follows now from (\ref{last_s_ineq_sat}) if $c$ is sufficiently small. For sufficiently large $c$ no restriction follows since we have $0<k<1$ by Lemma \ref{lemma_kc}.
\end{proof}

\begin{figure}
\epsfxsize=140mm
\begin{picture}(350,350)(0,0)
\put(-7.5,-8){\epsfbox{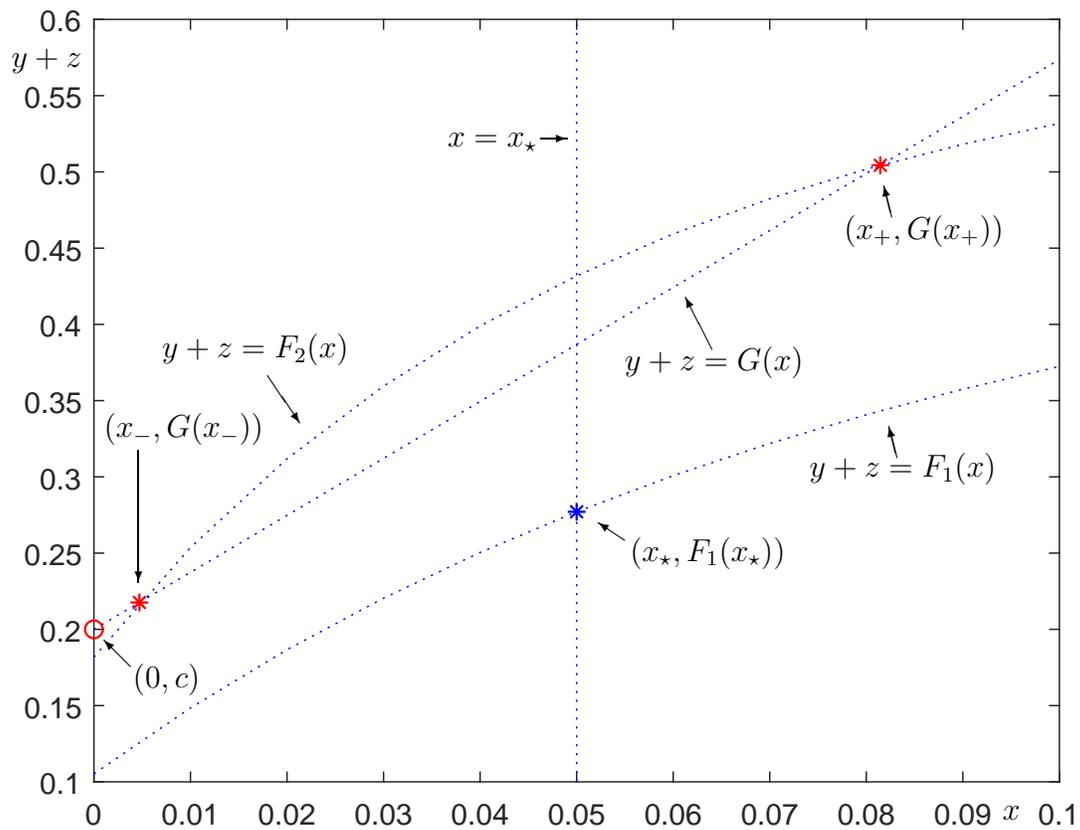}}
\put(34,45){$(0,c)$}
\put(33,53){\vector(-1,1){10}}

\put(23,140){$(x_-,G(x_-))$}
\put(36,135){\vector(0,-1){50}}

\put(222,93){$(x_\star,F_1(x_\star))$}
\put(220,101){\vector(-3,2){10}}
\put(153,250){$x=x_\star$}
\put(188,253){\vector(1,0){10}}
\put(363,-5){$x$}
\put(-12,280){$y+z$}
\put(45,170){$y+z=F_2(x)$}
\put(85,163){\vector(2,-3){12}}
\put(220,165){$y+z=G(x)$}
\put(253,173){\vector(-1,2){10}}
\put(290,125){$y+z=F_1(x)$}
\put(323,132){\vector(-1,4){3}}

\put(303,215){$(x_+,G(x_+))$}
\put(321,223){\vector(-1,4){3}}
\end{picture}
\caption{An interior fixed point exists if and only if $F_1(x_\star)<G(x_\star)<F_2(x_\star)$. Situation sketch for $c=.2$, $k=.95$, $x_\star=.05$,
$a_1=8.5$, $b_1=50$, $a_2=4.5$, $b_2=55$.}
\label{F2_G_F1}
\end{figure}
\begin{lemma}
It holds that $\lim_{x_\star\rightarrow 1}\breve{a}_2=\lim_{x_\star\rightarrow 1}\tilde{a}_2=k(1-c)=\check{a}_2$.
\label{limit_lemma}
\end{lemma}
\begin{proof}
The computation of the limit for $\breve{a}_2$ is straightforward after cancelling of the common factor $1+a_1/(1+b_1)$. The computation of the limit for $\tilde{a}_2$ leads to the square
\begin{displaymath}
\left(1+\frac{k}{1+b_2}-\frac{kc}{1+b_2}\right)^2
\end{displaymath}
under the radical and consequently, since $0<c<1$, to a straightforward computation.
\end{proof}

\begin{lemma}
If $x_\star>c$, then it holds that $\tilde{a}_2\leq\frac{1-c}{c}=\hat{a}_2$.
\label{lemma4}
\end{lemma}
\begin{remark}\em
We always have that $\check{a}_2=k(1-c)<(1-c)/c=\hat{a}_2$.
\em\end{remark}
\begin{proof}
It follows from the estimate
%\begin{displaymath}
%\frac{u}{2|w|}\left(1-\frac{u}{4w^2}\right)\leq-w+\sqrt{w^2+u}\leq\frac{u}{2|w|}
%\end{displaymath}
\begin{equation}
-w+\sqrt{w^2+u}\leq\frac{u}{2w},\:\:u\geq -w^2,\: w>0
\label{sqrt-estimate}
\end{equation}
with
\begin{displaymath}
w=\frac{(1-\frac{k}{1+b_2x_\star})x_\star+\frac{kc}{1+b_2x_\star}}{\frac{2x_\star}{1+b_2x_\star}},\:\: u=\frac{k(1-c)(1+b_2x_\star)}{x_\star}
\end{displaymath}
that
\begin{displaymath}
\tilde{a}_2\leq\frac{k(1-c)}{(1-\frac{k}{1+b_2x_\star})x_\star+\frac{kc}{1+b_2x_\star}}\leq\frac{k(1-c)}{c}\leq\frac{1-c}{c}=\hat{a}_2.
\end{displaymath}
We have a chain of three inequalities in the above expression and justify each of them separately. The first inequality holds by (\ref{sqrt-estimate}), the second since the expression in the denominator is a mean between $x_\star$ and $c$ and by the assumption $x_\star>c$. The last inequality holds since $k\leq 1$.
\end{proof}

\begin{lemma}
If $x_\star>(<)(a_1-\hat{a}_2)/(b_1\hat{a}_2)$, then $\check{a}_2<(>)\breve{a}_2$.
\label{lemma5}
\end{lemma}
\begin{proof} The inequality $\check{a}_2<(>)\breve{a}_2$ is equivalent to
the inequality
\begin{displaymath}
\left(1-c-\frac{a_1c}{1+b_1x_\star}\right)(1-x_\star)>(<)0
\end{displaymath}
The second factor is always positive. The sign of the first factor depends on $x_\star$. For large $x_\star$ it is positive and the limit for $x_\star$ is the one given in the assumption. If $a_1<\hat{a}_2$, then all values $0<x_\star<1$ gives $\check{a}_2<\breve{a}_2$.
\end{proof}
Combining Lemma \ref{lemma4} and \ref{lemma5}, we get
\begin{corollary}
If $x_\star>\max(c,(a_1-\hat{a}_2)/(b_1\hat{a}_2))$ and $0<k<\min(1,\hat{k})$ then $\emptyset\neq(\breve{a}_2,\tilde{a}_2)\subset(\check{a}_2,\hat{a}_2)$
\end{corollary}

\section{On the number of competition equilibria}
\label{number_eq}

We return for deriving the conditions for the existence of a certain number of competition equilibria. Define
\begin{equation}
\check{x}=\frac{-a_2^2-a_2(1-k)+b_2k(1-c)}{2a_2b_2}
\label{min_pt_q}
\end{equation}
and consider the quadratic equation
\begin{eqnarray}
   -a_2^2-a_2(1-k)+b_2k(1-c)&=&-(a_2-a_2^+)(a_2-a_2^-)=0,
\end{eqnarray}
along with the quartic equation
\begin{equation}
q_\ast(a_2)=4a_2b_2q(\check{x})=4a_2b_2kc\left(a_2-\frac{1-c}{c}\right)-((a_2-a_2^+)(a_2-a_2^-))^2=0.
\label{4e_gradare}
\end{equation}
We note that
\begin{displaymath}
a_2^\pm=\frac{-(1-k)\pm\sqrt{(1-k)^2+4b_2k(1-c)}}{2},
\end{displaymath}
$a_2^+>0$, $a_2^-<0$ and begin our analysis with the following lemma.
\begin{lemma}
Assume that $a_2^+>(1-c)/c$. In each of the intervals $]-\infty,a_2^-]$, $[a_2^-,0]$, $[(1-c)/c,a_2^+]$, and $[a_2^+,\infty]$, the equation \em (\ref{4e_gradare}) \em has exactly one real root.
\label{interval_root_lemma}
\end{lemma}
\begin{proof}
We have
\begin{eqnarray*}
\lim_{a_2\rightarrow -\infty}q_\ast(a_2)&=&-\infty,\\
  q_\ast(a_2^-)&>&0,  \\
  q_\ast(0)&<&0,  \\
  q_\ast\left(\frac{1-c}{c}\right) &<&0,  \\
  q_\ast(a_2^+)&>&0, \:\:{\rm{and}}\\
  \lim_{a_2\rightarrow\infty}q_\ast(a_2)&=&-\infty,
\end{eqnarray*}
implying that the specified location of the roots follows by the intermediate value theorem.
\end{proof}
We denote the root of (\ref{4e_gradare}) in the interval $[(1-c)/c,a_2^+]$ by $a_2^\ast$ and put
\begin{displaymath}
\underline{a}_2=\frac{-(1-k+b_2)+\sqrt{(1-k+b_2)^2+4b_2k(1-c)}}{2}.
\end{displaymath}
We can now state the following lemma. It simplifies the formulation of the next theorem and the range of possible dynamical scenarios for our competition system (\ref{competitive_system_auto_mixo}). Indeed, for increasing $a_2$ an additional competition equilibrium can bifurcate from the mixotroph carrying capacity $(0,0,c)$ through a transcritical bifurcation only.
\begin{lemma}
We have that
\begin{displaymath}
\underline{a}_2<\frac{1-c}{c}=\hat{a}_2.
\end{displaymath}
\label{lemma12}
\end{lemma}
\begin{proof}
We shall use (\ref{sqrt-estimate}) for proving the inequality. We have with $w=(1-k+b_2)/2$ and $u=b_2k(1-c)$ that it is enough to show
\begin{displaymath}
\underline{a}_2<\frac{b_2k(1-c)}{1-k+b_2}<\frac{1-c}{c}=\hat{a}_2
\end{displaymath}
The first inequality holds by (\ref{sqrt-estimate}). It remains to prove the second inequality. It is equivalent to
\begin{displaymath}
-(1-k)<b_2(1-kc)
\end{displaymath}
which is identically true since its left-hand side is negative and its right-hand side is positive.
\end{proof}
We now have the following bi-stability theorem.
\begin{Theorem}
If $k(1-c)<a_2<(1-c)/c$ then \em (\ref{sat_comp_altII}) \em has one interior equilibrium. If
\begin{equation}
\frac{1-c}{c}<a_2<a_2^\ast
\label{2_fixed_int_comp}
\end{equation}
then \em (\ref{sat_comp_altII}) \em has two interior equilibria. Otherwise, \em (\ref{sat_comp_altII}) \em has no interior equilibria.
\end{Theorem}

\begin{proof}
Consider the quadratic equation $q(x)=0$ given by
(\ref{multiple_eq_alt_II}). The condition is $q(0)q(1)<0$ becomes
\begin{displaymath}
kc(a_2+b_2+1)\left(a_2-\frac{1-c}{c}\right)(a_2-k(1-c))<0.
\end{displaymath}
Thus, we have precisely one competition equilibrium when
\begin{displaymath}
\check{a}_2=k(1-c)<a_2<\frac{1-c}{c}=\hat{a}_2.
\end{displaymath}
We notice that the condition for the presence of precisely one competition equilibrium is independent from saturation. The first statement of our bi-stability theorem follows.

For the second statement, the first necessary condition is that the function $q$ must have either a local minimum or a local maximum in the unit interval. Since the leading term of $q$ is positive, only local minima are possible. A minimum requires
\begin{eqnarray*}
  q(0) &=& kc\left(a_2-\frac{1-c}{c}\right)>0 \\
  q(1) &=& (a_2+b_2+1)(a_2-k(1-c))>0
\end{eqnarray*}
or
\begin{displaymath}
  a_2>\frac{1-c}{c}>k(1-c).
\end{displaymath}
The minimum point $\check{x}$ of $q$ was defined by (\ref{min_pt_q})
and we require $0<\check{x}<1$. That is
\begin{eqnarray*}
   -a_2^2-a_2(1-k)+b_2k(1-c)&>&0,  \\
   -a_2^2-a_2(1-k+b_2)+b_2k(1-c)&<&0,
\end{eqnarray*}
which in terms of the mixotrophic parameter $a_2$ can be stated as $\underline{a}_2<a_2<a_2^+$.

Finally, we need $q(\check{x})<0$. It follows from (\ref{4e_gradare}) and Lemma \ref{interval_root_lemma} that the additional requirement $(1-c)/c<a_2<a_2^\ast$ must hold. By Lemma \ref{lemma12}, (\ref{sat_comp_altII}) has two equilibria when (\ref{2_fixed_int_comp}) holds.
\end{proof}
The above theorem does not specify what saturation levels actually are needed in order to grant an interval of type (\ref{2_fixed_int_comp}). The following necessary condition turned out to be helpful.
\begin{corollary}
If an interval of type \em (\ref{2_fixed_int_comp}) \em exist, then
\begin{equation}
cb_2>\frac{1}{k}-1.
\label{ness_crit}
\end{equation}
\end{corollary}
\begin{proof}
In order to have an interval of type (\ref{2_fixed_int_comp}) we need at least $a_2^+>(1-c)/c$. From (\ref{sqrt-estimate}) we have
\begin{displaymath}
a_2^+=\frac{k-1+\sqrt{(k-1)^2+4b_2k(1-c)}}{2}<\frac{b_2k(1-c)}{1-k}
\end{displaymath}
and a necessary criterion for two competition equilibria is therefore (\ref{ness_crit}).
\end{proof}
Multiple competition equilibria are therefore, expected for $k$ close to one and $cb_2$ large. Remember that $0<k<1$, $0<c<1$, and $b_2>0$.
\begin{example}\em
In Figure \ref{faces_case1} we make a special study of the case $c=.2$, $k=.95$, $x_\star=.05$, $a_1=8.5$, $b_1=50$, $a_2=4.5$, and $b_2=55$. This case allows for two competition equilibria (a), and unstable equilibrium in the predator-prey plane (b), and a coexistence equilibrium, cf Figure \ref{F2_G_F1}. Consequently, the function $q$ has two zeros in the unit interval and a minimum point. The function $q_\ast$ has been plotted in (d). It has zeros in the prescribed intervals and in this case $4=(1-c)/c<a_2<a_2^\ast\approx 5.11<a_2^+\approx 6.44.$
\em\end{example}

\begin{figure}
\epsfxsize=166mm
\begin{picture}(350,350)(0,0)
\put(-42.5,-16){\epsfbox{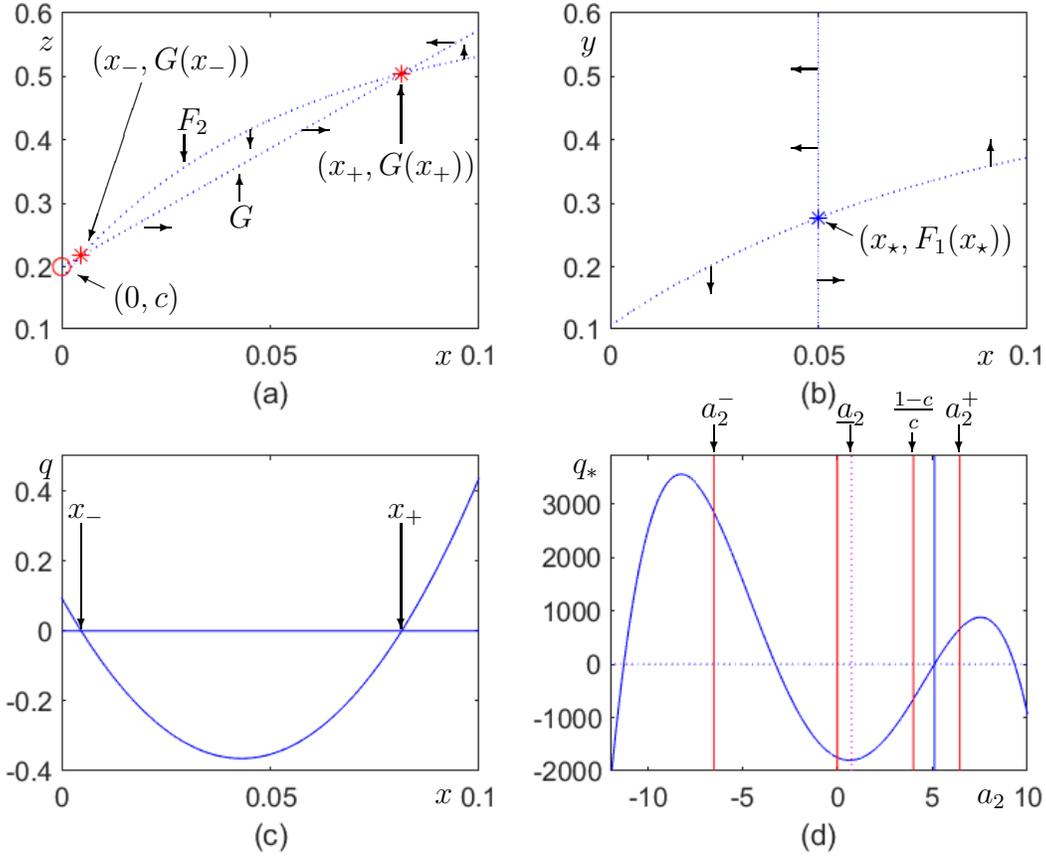}}
\put(10,297){$z$}
\put(160,178){$x$}
\put(82,230){$G$}
\put(86,240){\vector(0,1){10}}
\put(62,267){$F_2$}
\put(65,265){\vector(0,-1){10}}
\put(115,250){$(x_+,G(x_+))$}
\put(147,262){\vector(0,1){22}}
\put(30,290){$(x_-,G(x_-))$}
\put(49,285){\vector(-1,-3){20}}
\put(38,200){$(0,c)$}
\put(35,207){\vector(-2,1){10}}
\put(50,230){\vector(1,0){10}}
\put(110,267){\vector(1,0){10}}
\put(90,267){\vector(0,-1){7}}
\put(167,300){\vector(-1,0){10}}
\put(171,294){\vector(0,1){5}}

\put(215,297){$y$}
\put(365,178){$x$}
\put(320,222){$(x_\star,F_1(x_\star))$}
\put(318,227){\vector(-2,1){10}}
\put(304.5,210){\vector(1,0){10}}
\put(370,253){\vector(0,1){10}}
\put(304.5,260){\vector(-1,0){10}}
\put(304.5,290){\vector(-1,0){10}}
\put(264.2,215){\vector(0,-1){10}}

\put(160,12){$x$}
\put(21,120){$x_-$}
\put(26,118){\vector(0,-1){40}}
\put(142,120){$x_+$}
\put(147,118){\vector(0,-1){40}}
\put(10,137){$q$}

\put(365,12){$a_2$}
\put(212,137){$q_\ast$}
\put(260.5,158){$a_2^-$}
\put(265.5,155){\vector(0,-1){10}}

\put(312,158){$\underline{a}_2$}
\put(317,155){\vector(0,-1){10}}
\put(332.5,158){$\frac{1-c}{c}$}
\put(340.5,152){\vector(0,-1){7}}

\put(353,158){$a_2^+$}
\put(358,155){\vector(0,-1){10}}
\end{picture}
\caption{Special study of the case $c=.2$, $k=.95$, $x_\star=.05$, $a_1=8.5$, $b_1=50$, $a_2=4.5$, and $b_2=55$. (a) The functions $F_2$ and $G$ intersect at two competition equilibrium points in the $xz$-plane. (b) The predator-prey plane has an unstable equilibrium and possesses a unique limit cycle. (c) The function $q$ has two zeros in the unit interval and a minimum point. (d) The function $q_\ast$ has zeros in the prescribed intervals.}
\label{faces_case1}
\end{figure}

\section{The stability of the boundary equilibria}
\label{stab_bound_eq}

The generic Jacobian of (\ref{chemostat_mixo_sat_isocline_altII}) takes the form
\begin{displaymath}
J(x,y,z)=\left(\begin{array}{ccc}
\begin{array}{c} f_1^\prime(x)(F_1(x)-y)\\+f_1(x)F_1^\prime(x)-zf_2^\prime(x)\end{array} & -f_1(x) & -f_2(x)\\
y\psi^\prime(x) & \psi(x) & 0\\
zG^\prime(x) & -z & G(x)-y-2z
\end{array}\right).
\end{displaymath}
In the following we calculate the Jacobian for the different equilibrium points:
\begin{enumerate}
\item[(i)] The washout Jacobian $J(0,0,0)$ takes a diagonal form with the diagonal elements $1/k>0$, $-a_1x_\star/(k(1+b_1x_\star))<0$ and $c>0$. Therefore this equilibrium is always a saddle.
\item[(ii)] Similarly, the carrying capacity equilibrium takes an upper triangular form with diagonal $-1/k<0$, $\psi(1)>0$, $G(1)$. It has a transcritical bifurcation at
\begin{displaymath}
a_2=k(1-c)
\end{displaymath}
corresponding to the sign-change of $G(1)$. Hence this equilibrium is always a saddle.
\item[(iii)] The mixotroph carrying capacity equilibrium $(0,0,c)$ takes lower triangular form with diagonal elements
\begin{displaymath}
(1-(1+a_2)c)/k,\: -a_1x_\star/(k(1+b_1x_\star))<0,\: -c<0
\end{displaymath}
and is a saddle if $a_2<(1-c)/c$ and a stable node if $a_2>(1-c)/c$.
\item[(iv)] The Jacobian at the predator-prey equilibrium $(x_\star,F_1(x_\star),0)$ takes the block-diagonal form
\begin{displaymath}
J(x_\star,F_1(x_\star),0)=\left(\begin{array}{ccc}
                                 f_1(x_\star)F_1^\prime(x_\star) & -f_1(x_\star) & -f_2(x_\star) \\
                                 F_1(x_\star)\psi^\prime(x_\star) & 0 & 0 \\
                                  0 & 0 & G(x_\star)-F_1(x_\star)
                                \end{array}\right)
\end{displaymath}
with a $2\times 2$ block matrix reflecting the stability properties in the predator-prey plane that are already known and a single eigenvalue reflecting the behavior in the mixotroph direction of magnitude $G(x)-F_1(x)$. Criterion (\ref{iso_interior_crit}) is visible here.

\item[(v)] We are now ready to proceed to the possible competition equilibria at $(x_\pm,0,G(x_\pm))$. Now we consider the Jacobian
\begin{equation}
J(x_\pm,0,G(x_\pm))=\left(\begin{array}{ccc}
f_2(x_\pm)F_2^\prime(x_\pm) & -f_1(x_\pm) & -f_2(x_\pm)\\
0 & \psi(x_\pm) & 0\\
G(x_\pm)G^\prime(x_\pm) & -G(x_\pm) & -G(x_\pm)
\end{array}\right).
\label{Jaco_comp_eq}
\end{equation}
The corner elements form a $2\times 2$-block matrix reflecting the stability properties in the competition plane in terms of the intersections of $F_2$ and $G$. Indeed, the determinant of the corner elements is $f_2(x_\pm)G(x_\pm)(G^\prime(x_\pm)-F_2^\prime(x_\pm))$. Since $G^\prime(x_-)-F_2^\prime(x_-)<0$, it follows that $(x_-,0,G(x_-))$ is a saddle in the $xz$-plane if it exists. Similarly, $(x_+,0,G(x_+))$ is never a saddle in the $xz$-plane. Consequently $(x_+,0,G(x_+))$ is stable in the $xz$-plane if $f_2(x_+)F_2^\prime(x_+)-F_2(x_+)<0$ (Note that $F_2(x_+)=G(x_+)$). The last eigenvalue of the competition equilibria is $\psi(x_\pm)$. The second inequality of (\ref{iso_interior_crit}) implies $\max(0,x_-)<x_\star<x_+$ and therefore we always have invasion of herbivores in the vicinity of $(x_+,0,F_2(x_+))$ if $(x_\star,y_\star,z_\star)$ exists.
\end{enumerate}
Our program is now to eliminate some regions that are of limited interest for a subsequent numerical study. We begin by the following theorem.
\begin{Theorem}
Consider the competition system \em (\ref{sat_comp_altII}). \em If $G$ decreases, then all solutions converge towards an equilibrium.
\label{no_comp_cycles_G}
\end{Theorem}
\begin{proof}
We note that the Jacobian of (\ref{sat_comp_altII}) has the off-diagonal elements
\begin{displaymath}
-f_2(x)<0\:{\rm{and}}\: zG^\prime(x)<0
\end{displaymath}
meaning that the system is competitive, ie its solutions are backwards monotone. All solutions of a backwards monotone two-dimensional system converge towards an equilibrium, see e. g. Smith (1995)\nocite{Smith.monotone}.
\end{proof}
The next theorem excludes cycles in the competition plane for low saturation levels. In particular, the unsaturated competition system (\ref{sat_comp_altII}) with $b_2=0$ cannot possess competition cycles. This excludes a large parameter region that becomes of limited interest for the subsequent numerical part in the end of this paper.
\begin{Theorem}
Consider the competition system \em (\ref{sat_comp_altII}). \em If $F_2$ decreases, then all solutions converge towards an equilibrium.
\label{no_comp_cycles}
\end{Theorem}
\begin{proof}
Consider (\ref{sat_comp_altII}) and the Dulac function $B(x,z)=\frac{1}{zf_2(x)}$. The divergence of the system
\begin{eqnarray*}
  \dot{x} &=& \frac{F_2(x)}{z}-1 \\
  \dot{z} &=& \frac{G(x)-z}{f_2(x)}
\end{eqnarray*}
can now be estimated as
\begin{displaymath}
\frac{F_2^\prime(x)}{z}-\frac{1}{f_2(x)}<-\frac{1}{f_2(x)}<0.
\end{displaymath}
By Dulac's theorem (Brauer and Castillo-Ch\'{a}vez (2001)\nocite{Brauer.Mathmod} and Ye (1986)\nocite{ye}), (\ref{sat_comp_altII}) cannot possess periodic orbits. By the Poincar\'{e}-Bendixson's theorem (Hirsch et al. (2013)\nocite{hirschsmaledevaney} and Wiggins (2003)\nocite{wiggins2003}), possible limit sets are equilibria or orbits connecting equilibria. The origin is not a saddle and cannot belong to a limit set consisting of equilibria and orbits connecting them. If a saddle exist in the interior of the positive quadrant, then $G$ must decrease between $x_-$ and $x_+$. In this case, $G$ decreases on the unit interval and the solutions converge towards an equilibrium by Theorem \ref{no_comp_cycles_G}. Assume therefore, that the saddles of (\ref{sat_comp_altII}) are located either at the $x$-axis or at the $z$-axis. In both cases, either its stable or unstable manifolds are located along the axes. Possible saddle-connections violate therefore, uniqueness of solutions. Consequently, all solutions converge towards an equilibrium.
\end{proof}

\begin{figure}
\epsfxsize=138mm
\begin{picture}(350,350)(0,0)
\put(0,0){\epsfbox{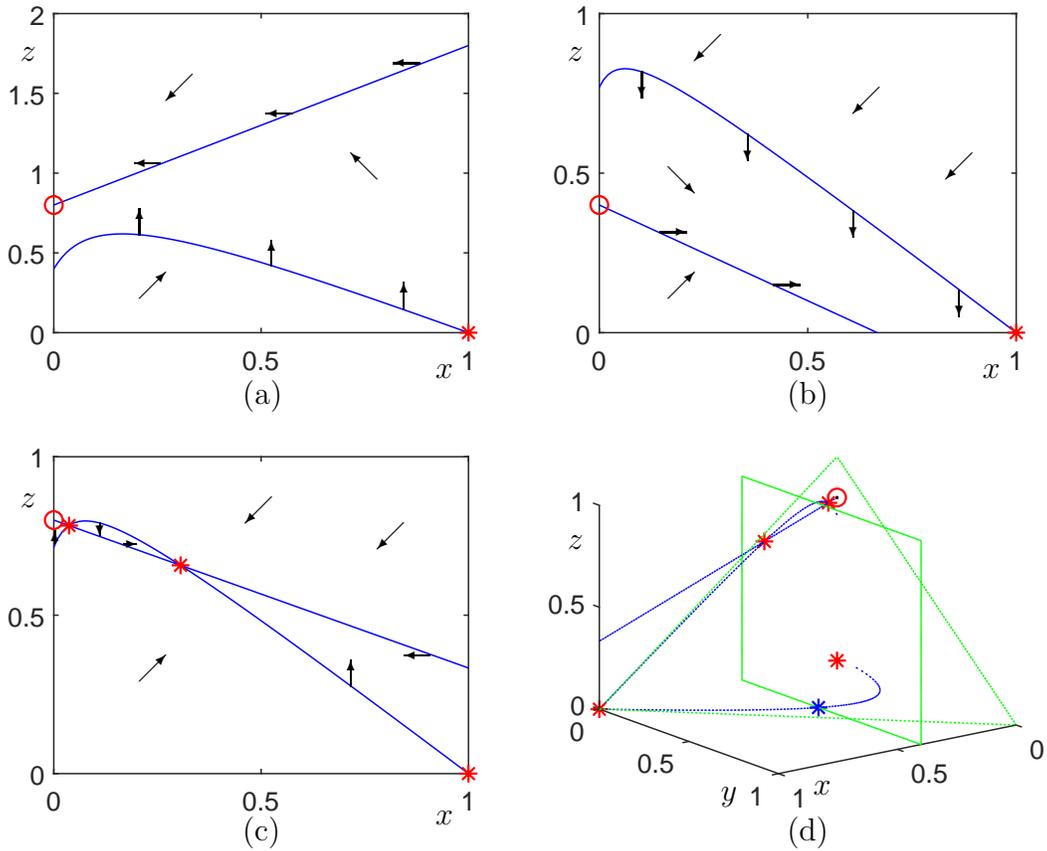}}
\put(89,155){(a)}
\put(5,285){$z$}
\put(162,165){$x$}
\put(50,219){\vector(0,1){10}}
\put(100,207){\vector(0,1){10}}
\put(150,191){\vector(0,1){10}}
\put(58,246){\vector(-1,0){10}}
\put(108,265){\vector(-1,0){10}}
\put(156,284){\vector(-1,0){10}}
\put(140,240){\vector(-1,1){10}}
\put(70,280){\vector(-1,-1){10}}
\put(50,195){\vector(1,1){10}}
\put(295,155){(b)}
\put(212,285){$z$}
\put(369,165){$x$}
\put(290,200){\vector(1,0){10}}
\put(247,220){\vector(1,0){10}}
\put(240,281){\vector(0,-1){10}}
\put(280,257){\vector(0,-1){10}}
\put(320,228){\vector(0,-1){10}}
\put(360,198){\vector(0,-1){10}}
\put(250,195){\vector(1,1){10}}
\put(250,245){\vector(1,-1){10}}
\put(270,295){\vector(-1,-1){10}}
\put(330,275){\vector(-1,-1){10}}
\put(365,250){\vector(-1,-1){10}}
\put(89,-10){(c)}
\put(5,116){$z$}
\put(162,-4){$x$}
\put(160,60){\vector(-1,0){10}}
\put(130,48){\vector(0,1){10}}
\put(35,110){\vector(0,-1){5}}
\put(44,102){\vector(1,0){5}}
\put(18,102){\vector(0,1){5}}
\put(50,50){\vector(1,1){10}}
\put(150,110){\vector(-1,-1){10}}
\put(100,120){\vector(-1,-1){10}}
\put(295,-10){(d)}
\put(212,100){$z$}
\put(270,5){$y$}
\put(305,7){$x$}
\end{picture}
\caption{(a) Phase-plane arguments for globally attracting $(0,c)$ when $G>F_2$ in the case $c=.4$, $k=.75$, $a_2=.3$, and $b_2=20$. (b) Phase-plane arguments for globally attracting $(1,0)$ when $F_2>G$ in the case $c=.8$, $k=.75$, $a_2=1.5$, and $b_2=20$. (c) Multiple attractors in the competition plane for $c=.8$, $k=.75$, $a_2=.4$, and $b_2=20$. (d) Multiple attractors persist in system (\protect\ref{chemostat_mixo_sat_scaled_altII}) for $c=.8$, $k=.75$, $x_\star=.4$, $a_1=8.5$, $b_1=50$, $a_2=.4$, and $b_2=20$. The equilibria $(0,0,0)$ (red *-mark), $(1,0,0)$ (red *-mark), $(x_\star, F_1(x_\star),0)$ (blue *-mark), $(x_+,0,F_2(x_+))$ (red *-mark), $(x_-,0,F_2(x_-))$ (red *-mark), and $(0,0,c)$ (red o-mark) exist and are marked. The boundary of the simplex $x+y+z\leq 1$ and the plane $x=x_\star$ are marked in green.}
\label{multiple_attractors}
\end{figure}

We eliminate the last large region of limited interest by the following theorem.
\begin{Theorem}
Consider the competition system \em (\ref{sat_comp_altII}). \em If no competition equilibria exist, then all solutions converge towards an equilibrium.
\end{Theorem}
\begin{proof}
Since, $F_2$ and $G$ have no intersections, we have either that $G>F_2$, or $F_2>G$. In each of the cases, and all positive solutions converge towards $(0,c)$ or towards $(1,0)$, respectively, by phase-plane arguments, see Figure \ref{multiple_attractors}(a)-(b).
\end{proof}
We conclude this section by proving that Theorem \ref{no_comp_cycles_G} implies the presence of parameter values giving rise to multiple attractors. Indeed, the parameter values $c=.8$, $k=.75$, $a_2=.4$, and $b_2=20$ gives rise to multiple attractors for the competition system (\ref{sat_comp_altII}), see Figure \ref{multiple_attractors}(c). The equilibria $(0,c)$ and $(x_+,F(x_+))$ are both locally attracting and no cycles exist by Theorem \ref{no_comp_cycles_G}. This situation persists for (\ref{chemostat_mixo_sat_scaled_altII}) if $x_\star>x_+$ is selected. The situation is depicted in Figure \ref{multiple_attractors}(d) for $c=.8$, $k=.75$, $x_\star=.4$, $a_1=8.5$, $b_1=50$, $a_2=.4$, and $b_2=20$. The equilibria $(0,0,c)$ and $(x_+,0,F_2(x_+))$ are still both locally stable.

\section{The stability of the coexistence equilibrium}
\label{stab_int_eq}

We proceed to the interior equilibrium $(x_\star,y_\star,z_\star)$. Its Jacobian matrix is given by
\begin{displaymath}
J(x_\star,y_\star,z_\star)=\left(\begin{array}{ccc}
\begin{array}{c} f_1(x_\star)F_1^\prime(x_\star)\\-z_\star\left(f_2^\prime(x_\star)-f_1^\prime(x_\star)\frac{f_2(x_\star)}{f_1(x_\star)}\right)\end{array} & -f_1(x_\star) & -f_2(x_\star)\\
y_\star\psi^\prime(x_\star) & 0 & 0\\
z_\star G^\prime(x_\star) & -z_\star  & -z_\star
\end{array}\right).
\end{displaymath}
with
\begin{displaymath}
y_\star=\frac{f_2(x_\star)(F_2(x_\star)-G(x_\star))}{f_1(x_\star)-f_2(x_\star)},\:\:z_\star=\frac{f_1(x_\star)(G(x_\star)-F_1(x_\star))}{f_1(x_\star)-f_2(x_\star)}.
\end{displaymath}

In order to continue formulating stability conditions for this equilibrium, we first give the Routh (1877)\nocite{routh1}-Hurwitz (1895)\nocite{MA.Hurwitz:46} conditions (see also May (1974)\nocite{may} and Wiggins (2003)\nocite{wiggins2003}) in matrix form. The matrixes $A_{ii}$ are the submatrixes obtained when the $i$th row and column are deleted from $A$. Sometimes they are referred to as principal submatrices, see e. g. Strang (2006)\nocite{Strang.Linear}.
\begin{lemma}
The eigenvalues of the $3\times 3$ matrix $A$ have negative real parts if and only if \em (i) \em ${\rm{Tr}{A}}<0$, \em (ii) \em ${\rm{det}}A<0$, and
\begin{displaymath}
{\rm{(iii)}}\:\:{\rm{Tr}{A}}({\rm{det}}A_{11}+{\rm{det}}A_{22}+{\rm{det}}A_{33})-{\rm{det}}A<0.
\end{displaymath}
If some of the eigenvalues have zero real part, then equality must hold in either \em (ii) \em or \em (iii). \em In particular, if some of the eigenvalues are zero, then equality holds in \em (ii).\em
\label{RouthHurwitzMatrixForm}
\end{lemma}
\begin{proof} Identification of the coefficients of the characteristic polynomial gives the matrix form of the criterion. Since $\det A$ is the product of the eigenvalues, a zero eigenvalue gives equality in (ii). The next possibility for zero real parts is a real eigenvalue $\lambda_1=\lambda<0$ and complex conjugate pair $\lambda_{2,3}=\pm\omega i$. The relation between roots and coefficients implies that
\begin{eqnarray*}
  {\rm{Tr}}A &=& \lambda_1+\lambda_2+\lambda_3=\lambda<0\\
  {\rm{det}}A_{11}+{\rm{det}}A_{22}+{\rm{det}}A_{33} &=& \lambda_1\lambda_2+\lambda_2\lambda_3+\lambda_1\lambda_3 \\
  &=&\lambda\omega+\omega^2-\lambda\omega=\omega^2>0\\
  {\rm{det}}A &=& \lambda_1\lambda_2\lambda_3=\lambda\omega^2<0
\end{eqnarray*}
Criterion (iii) now becomes
\begin{displaymath}
{\rm{Tr}{A}}({\rm{det}}A_{11}+{\rm{det}}A_{22}+{\rm{det}}A_{33})-{\rm{det}}A=\lambda\omega^2-\lambda\omega^2=0.
\end{displaymath}
\end{proof}
\begin{corollary}
A real $3\times 3$-matrix $A$ cannot lose its stability by passing through an equality in \em (i) \em in Lemma \em \ref{RouthHurwitzMatrixForm} \em above.
\label{stability-loss}
\end{corollary}
\begin{proof}
If $\det A<0$ and ${\rm{Tr}}A=0$, then (iii) becomes
\begin{displaymath}
{\rm{Tr}{A}}({\rm{det}}A_{11}+{\rm{det}}A_{22}+{\rm{det}}A_{33})-{\rm{det}}A=-\det A>0
\end{displaymath}
contradicting the possibility for stability losses at ${\rm{Tr}}A=0$.
\end{proof}
\begin{remark}\em
Even if stability losses at ${\rm{Tr}}A=0$ are excluded by the above corollary, criterion (i) is still a necessary criterion for stability in Lemma \ref{RouthHurwitzMatrixForm}. Compare e. g. with the case $\lambda_1=-1$, $\lambda_2=\frac{1}{2}$, $\lambda_3=\frac{3}{2}$. We get
\begin{eqnarray*}
  {\rm{Tr}}A &=& \lambda_1+\lambda_2+\lambda_3=-1+\frac{1}{2}+\frac{3}{2}=1>0 \\
  \det A &=& \lambda_1\lambda_2\lambda_3=-1\cdot\frac{1}{2}\cdot\frac{3}{2}=-\frac{3}{4}<0  \\
  \det A_{11}+{\rm{det}}A_{22}+{\rm{det}}A_{33} &=& \lambda_1\lambda_2+\lambda_1\lambda_3+\lambda_2\lambda_3= -\frac{1}{2}-\frac{3}{2}+\frac{3}{4}=-\frac{5}{4}
\end{eqnarray*}
meaning that
\begin{displaymath}
{\rm{Tr}{A}}({\rm{det}}A_{11}+{\rm{det}}A_{22}+{\rm{det}}A_{33})-{\rm{det}}A=
1\cdot\left(-\frac{5}{4}\right)-\left(-\frac{3}{4}\right)=-\frac{1}{2}<0.
\end{displaymath}
We see that the associated Jacobian matrix is unstable and that criteria (ii) and (iii) in Theorem \ref{RouthHurwitzMatrixForm} are satisfied.
\em\end{remark}
We now formulate the stability criteria of the interior fixed point in terms of criteria needed in addition to the global stability condition of the the autotroph-herbivore equilibrium in the $xy$-plane.
\begin{Theorem}
The coexistence equilibrium $(x_\star,y_\star,z_\star)$ is locally stable if
\begin{enumerate}
  \item [\em(i)\em] $F_1^\prime(x_\star)<0$
  \item [\em(ii)\em] $G^\prime(x_\star)>0$
  \item [\em(iii)\em] $f_1(x_\star)f_2^\prime(x_\star)-f_1^\prime(x_\star)f_2(x_\star)>0$
\end{enumerate}
The coexistence equilibrium loses its stability through a Hopf-bifurcation, if it exists.
\label{complexity-stability-relation}
\end{Theorem}

\begin{remark}\em
The requirement that the herbivore-autotroph equilibrium is stable in the plane $z=0$ is a visible in condition (i). The requirement comes from Lemma \ref{glob_stab}. Criterion (ii) is easy to check, $a_2>k$ is sufficient.
\em\end{remark}

\begin{remark}\em
Criterion (iii) is strongly related to the local condition (\ref{local_condition_B}) and the upper bound for $a_2$ in (\ref{upper_bound_a2}). However, they are not equivalent and the ratio between $b_1$ and $b_2$ determines the direction of the implication between Criterion (iii) and (\ref{local_condition_B}). Criterion (iii) can be modified into the equivalent parametric condition
\begin{displaymath}
a_1b_1-a_2b_2+a_1a_2(b_1-b_2)+2b_1b_2(a_1-a_2)x_\star-b_1b_2(a_2b_1-a_1b_2)x_\star^2>0.
\end{displaymath}
Lemma \ref{lemma_loc_cond} (iii) can now be used to obtain the estimate
\begin{eqnarray}
a_1b_1-a_2b_2+a_1a_2(b_1-b_2)+2b_1b_2(a_1-a_2)x_\star-b_1b_2(a_2b_1-a_1b_2)x_\star^2&>&\label{par_ineq}\\
a_1b_1-a_2b_2+a_1a_2(b_1-b_2)+b_1b_2(a_1-a_2)x_\star&>&0.\nonumber
\end{eqnarray}
We see that all terms in the last expression are positive for $b_1>b_2$. Indeed, (\ref{upper_bound_a2}) gives $a_2<a_1$ in this case. Therefore, (\ref{local_condition_B}) ensures Criterion (iii) in Theorem \ref{complexity-stability-relation} for $b_1>b_2$. If $b_1<b_2$, we have $a_2<a_1b_2/b_1$. We now use this relation for deriving a stricter but very similar condition as (\ref{upper_bound_a2}). The second inequality in (\ref{par_ineq}) gives
\begin{displaymath}
a_2<\frac{a_1b_1+b_1b_2a_1x_\star}{b_2+a_1(b_2-b_1)+b_1b_2x_\star}=a_1\frac{b_1}{b_2}\frac{1+b_2x_\star}{1+\frac{a_1}{b_2}(b_2-b_1)+b_1x_\star}<a_1\frac{1+b_2x_\star}{1+b_1x_\star}.
\label{a2_estimate}
\end{displaymath}
The last inequality compares this estimate with (\ref{upper_bound_a2}). Since our estimate increases with $x_\star$, worst cases can be computed with $x_\star=0$. In many cases such estimates makes criterion (iii) redundant in comparison to the essential criterion (i). The situation for $b_1=50$, $b_2=55$, $a_1=8.5$ is illustrated in Figure \ref{a2_limits}
\em\end{remark}
\begin{figure}
\epsfxsize=138mm
\begin{picture}(350,350)(0,0)
\put(0,0){\epsfbox{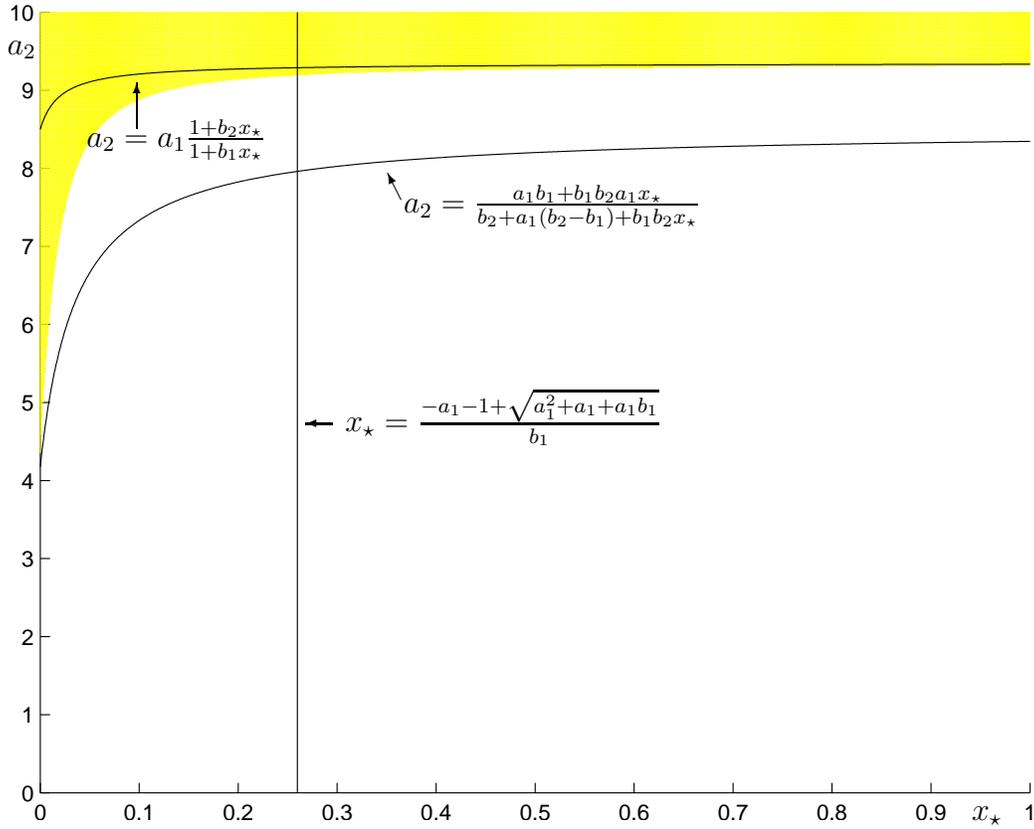}}
\put(365,0){$x_\star$}
\put(123,150){\vector(-1,0){10}}
\put(128,147){$x_\star=\frac{-a_1-1+\sqrt{a_1^2+a_1+a_1b_1}}{b_1}$}
\put(149,235){\vector(-1,2){5}}
\put(150,230){$a_2=\frac{a_1b_1+b_1b_2a_1x_\star}{b_2+a_1(b_2-b_1)+b_1b_2x_\star}$}
\put(49,262){\vector(0,1){17}}
\put(30,255){$a_2=a_1\frac{1+b_2x_\star}{1+b_1x_\star}$}
\put(0,290){$a_2$}
\end{picture}
\caption{Criterion (iii) of Theorem \protect\ref{complexity-stability-relation} in $(x_\star,a_2)$ parameter plane for the case $a_1=8.5$, $b_1=50$, and $b_2=55$ in comparison to the curves defined by (\protect\ref{upper_bound_a2}) and the first inequality of (\protect\ref{a2_estimate}). Criterion (iii) of Theorem \protect\ref{complexity-stability-relation} is invalid in the yellow area. Criterion (i) is not valid on the left-hand side of the vertical line showing that the largest deviances from (\protect\ref{upper_bound_a2}) occur in areas where (i) is not valid anyway.}
\label{a2_limits}
\end{figure}
We proceed with the proof of Theorem \ref{complexity-stability-relation}.
\begin{proof}
We use Lemma \ref{RouthHurwitzMatrixForm} to derive stability conditions for the interior equilibrium. Indeed, we have
\begin{displaymath}
{\rm{Tr}}J(x_\star,y_\star,z_\star)=f_1(x_\star)F_1^\prime(x_\star)-z_\star\left(f_2^\prime(x_\star)-f_1^\prime(x_\star)\frac{f_2(x_\star)}{f_1(x_\star)}\right)-z_\star<0.
\end{displaymath}
Next
\begin{eqnarray*}
{\rm{det}}J(x_\star,y_\star,z_\star)&=&-y_\star\psi^\prime(x_\star)\left|\begin{array}{cc}
-f_1(x_\star) & -f_2(x_\star)\\
-z_\star & -z_\star\end{array}\right|\\
&=&-y_\star z_\star\psi^\prime(x_\star)\left(f_1(x_\star)-f_2(x_\star)\right)<0.
\end{eqnarray*}
The last inequality holds because our local condition (\ref{local_condition_B}). It follows that the coexistence equilibrium cannot have zero eigenvalues and can thus, lose its stability through Hopf-bifurcations only. In order to formulate condition (iii) of Lemma \ref{RouthHurwitzMatrixForm}, we first compute the principal subdeterminants
\begin{eqnarray*}
{\rm{det}}J_{11}(x_\star,y_\star,z_\star)&=&0,\\
{\rm{det}}J_{22}(x_\star,y_\star,z_\star)&=&z_\star\biggl(f_2(x_\star)G^\prime(x_\star)-f_1(x_\star)F_1^\prime(x_\star)\\&&+z_\star \left(f_2^\prime(x_\star)-f_1^\prime(x_\star)\frac{f_2(x_\star)}{f_1(x_\star)}\right)\biggr),\\
{\rm{det}}J_{33}(x_\star,y_\star,z_\star)&=&y_\star f_1(x_\star)\psi^\prime(x_\star),
\end{eqnarray*}
and note that (iii) of Lemma \ref{RouthHurwitzMatrixForm} is equivalent to
\begin{eqnarray*}
  {\rm{Tr}}J(x_\star,y_\star,z_\star)\det J_{22}(x_\star,y_\star,z_\star)+&&\\+{\rm{Tr}}J(x_\star,y_\star,z_\star)\det J_{33}(x_\star,y_\star,z_\star)-\det J(x_\star,y_\star,z_\star)=&& \\
  {\rm{Tr}}J(x_\star,y_\star,z_\star)\det J_{22}(x_\star,y_\star,z_\star)+&&\\
  +f_1(x_\star)F_1^\prime(x_\star)y_\star f_1(x_\star)\psi^\prime(x_\star)-z_\star y_\star f_1(x_\star)\psi^\prime(x_\star)\left(f_2^\prime(x_\star)-f_1^\prime(x_\star)\frac{f_2(x_\star)}{f_1(x_\star)}\right)&&\\-z_\star y_\star f_1(x_\star)\psi^\prime(x_\star)+y_\star z_\star\psi^\prime(x_\star)\left(f_1(x_\star)-f_2(x_\star)\right)=&&\\
  \underbrace{{\rm{Tr}}J(x_\star,y_\star,z_\star)}_{<0}\det J_{22}(x_\star,y_\star,z_\star)\underbrace{-y_\star z_\star f_2(x_\star)\psi^\prime(x_\star)}_{<0}+&&\\
  \underbrace{y_\star F_1^\prime(x_\star)(f_1(x_\star))^2\psi^\prime(x_\star)}_{<0}\underbrace{-y_\star z_\star  \psi^\prime(x_\star)(f_1(x_\star)f_2^\prime(x_\star)-f_1^\prime(x_\star)f_2(x_\star))}_{<0}.&&\\
\end{eqnarray*}
The whole quantity is therefore, negative when $\det J_{22}(x_\star,y_\star,z_\star)>0$. Indeed, we have since $F_1^\prime(x_\star)<0$
\begin{eqnarray*}
\frac{f_1(x_\star)}{z_\star}\det J_{22}(x_\star,y_\star,z_\star)&>&f_1(x_\star)f_2(x_\star)G^\prime(x_\star)-(f_1(x_\star))^2F_1^\prime(x_\star)\\&&+z_\star \left(f_1(x_\star)f_2^\prime(x_\star)-f_1^\prime(x_\star)f_2(x_\star)\right)>0.
\end{eqnarray*}

\end{proof}

\section{Numerical Results}
\label{numres}

According to Theorem \ref{complexity-stability-relation} existence of cycles in the predator-prey plane is an important criterion for an unstable coexistence equilibrium whenever such an equilibrium exist. In this section we fix the parameters as follows $c=.2$, $b_1=50$, $b_2=55$, $k=.95$, and $a_1=8.5$. The parameter values selected here allows an interval of existence for the coexistence equilibrium for all parameter values of $x_\star$, $0<x_\star<1$ according to Lemma \ref{exist-param_sat}. The upper bound for the efficiency of the mixotroph specified by (\ref{upper_bound_a2}) ($a_2\approx a_1=8.5$) does not enter to the window of parameter values selected for this diagram, cf Figure \ref{a2_limits}.

The major bifurcations of (\ref{chemostat_mixo_sat_scaled_altII}) are now indicated in Figure \ref{bifurcation_diagram}. Here we vary (i) the equilibrium value of $x_\star$ of the predator-prey equilibrium representing the interaction between the autotroph and the herbivore, and (ii) $a_2$ that determines the competitive ability between the mixotroph and the  autotroph for the limiting resources. Lemmas \protect\ref{glob_stab}-\protect\ref{unique_limit_cycle} gives the vertical line at $x_\star\approx.2598$. On the left-hand side of this line, the herbivore-autotroph system has a unique limit cycle (this corresponds to the areas {\bf{(a)}}, {\bf{(c)}}, {\bf{(d)}}, {\bf{(e)}}, {\bf{(g)}}, {\bf{(h)}}, {\bf{(i)}}, {\bf{(k)}}, {\bf{(l)}}, {\bf{(m)}}, {\bf{(p)}}, {\bf{(q)}}, {\bf{(r)}},  {\bf{(t)}}, {\bf{(u)}}, {\bf{(v)}}, {\bf{(w)}}, and {\bf{(x)}}) on the right hand side of this line $(x_\star,F_1(x_\star),0)$ is globally stable (this corresponds to the regions {\bf{(b)}}, {\bf{(f)}}, {\bf{(j)}}, {\bf{(n)}}, {\bf{(o)}}, {\bf{(s)}}, and {\bf{(y)}}).

In a similar manner, the horizontal lines correspond to dynamics in the competition plane. For low $a_2$ (in our case $a_2<k(1-c)=.76$), no competition equilibria exist and the carrying capacity equilibrium $(1,0,0)$ is stable in the competition plane (regions {\bf{(t)}}, {\bf{(u)}}, {\bf{(v)}}, {\bf{(w)}}, {\bf{(x)}}, and {\bf{(y)}}). For moderate values of $a_2$ (in our case $.76<a_2<3.582$) a unique competition equilibrium $(x_+,0,F_2(x_+))$ exists and is stable in the competition plane (regions {\bf{(k)}}, {\bf{(l)}}, {\bf{(m)}}, {\bf{(n)}}, {\bf{(o)}}, {\bf{(p)}}, {\bf{(q)}}, {\bf{(r)}}, and {\bf{(s)}}). For still larger values of $a_2$ (in our case $3.582<a_2<4$) a unique competition equilibrium exists, but no equilibria are stable in the competition plane (regions {\bf{(g)}}, {\bf{(h)}}, {\bf{(i)}}, {\bf{(j)}}). For further increase in $a_2$ ($4<a_2^\ast\approx 5.1118$) two competition equilibria exist and $(0,0,c)$ is stable in the competition plane (regions {\bf{(c)}}, {\bf{(d)}}, {\bf{(e)}}, and {\bf{(f)}}). For extreme values of $a_2$ ($a_2>a_2^\ast$), no competition equilibria exist and the mixotroph carrying capacity $(0,0,c)$ is stable in the competition plane (regions {\bf{(a)}} and {\bf{(b)}}).

The blue and red curves corresponds to the existence boundaries of the coexistence equilibrium. This means that the coexistence equilibrium exist in regions {\bf{(d)}}, {\bf{(g)}}, {\bf{(h)}}, {\bf{(k)}}, {\bf{(m)}}, {\bf{(n)}}, {\bf{(p)}}, {\bf{(t)}}, and {\bf{(u)}}. On the far right of the bifurcation diagram we confirm the conclusion of Lemma \ref{limit_lemma}. Below the blue curve, the coexistence equilibrium does not exist and at the blue curve it collides in a transcritical bifurcation with the predator-prey equilibrium. We see that the intersection of the blue curve with the horizontal line $a_2=k(1-c)=.76$ in the far left of the diagram  is in agreement with Lemma \ref{lemma5}. Above the red curve it does not exist, either. The red curve has a maximum at $a_2=a_2^\ast$. At the left hand side of this maximum, the coexistence equilibrium collides in a transcritical bifurcation with $(x_-,0,F_2(x_-))$ and at the right-hand side of this maximum a similar collision with $(x_+,0,F_2(x_+))$ occurs. The yellow area in regions {\bf{(g)}}, {\bf{(k)}}, {\bf{(u)}}, the cyan areas in regions {\bf{(p)}} and {\bf{(t)}}, the green areas in regions {\bf{(d)}}, {\bf{(h)}}, and possibly black areas on the boundaries of regions {\bf{(k)}}, {\bf{(g)}}, {\bf{(p)}}, and {\bf{(t)}}  corresponds to the region where the coexistence equilibrium is unstable. In the yellow area in regions {\bf{(k)}}, {\bf{(g)}}, and {\bf{(u)}} criterion (i) in Lemma \ref{RouthHurwitzMatrixForm} is sufficient for deducing instability whereas criterion (iii) is needed in the black areas. We see that this occurs on the boundaries only precisely as predicted by Corollary \ref{stability-loss}.

It is of importance to elucidate the dynamical behavior of our system in regions that possess no stable equilibria. In region {\bf{(s)}} and {\bf{(y)}}, the predator-prey equilibrium $(x_\star,F_1(x_\star),0)$ is stable. In regions {\bf{(m)}} and {\bf{(n)}}, the coexistence equilibrium $(x_\star,y_\star,z_\star)$ is stable. In regions {\bf{(l)}} and {\bf{(o)}}, the competition equilibrium $(x_+,0,F_2(x_+))$ is stable. In regions {\bf{(a)}}, {\bf{(b)}}, {\bf{(c)}}, {\bf{(d)}}, {\bf{(e)}}, and {\bf{(f)}}, the mixotroph carrying capacity $(0,0,c)$ is stable. For $x_\star>.2598$ region {\bf{(j)}} is of interest. There, the predator-prey plane has a fixed point that is globally stable in the plane but unstable with respect to mixotrophic invasion. However, at least one stable limit cycle exist in the competition plane and it is not clear whether or not it (or these limit cycles) might be unstable with respect to hervibore invasion. If cycles are small, they spend a lot of time near the competition equilibrium that has a negative third eigenvalue and when they grow larger they spend a substantial amount of time in the vicinity of the saddle $(0,0,c)$ that is still more stable with respect to herbivore invasion. We conclude that the invasion of mixotrophs has a destabilizing impact on the dynamics in region {\bf{(j)}} since the autotroph-herbivore dynamics results in global stability and invasion of mixotrophs leads to out-competition of the herbivores and resulting cyclic dynamics.

We selected the parameter values $x_\star=.26$ and $a_2=3.9$ for illustration of the destabilizing behavior in region {\bf{(j)}}. The parameter value is marked with a blue circle in region {\bf{(j)}} in our bifurcation diagram, Figure \ref{bifurcation_diagram}. The result is visible in Figure \ref{Simubild22}(a). The simplex containing all solutions and the $x=x_\star$-plane are marked with dotted green lines. The isocline in the $xy$-plane, $F_1$ is marked with a dotted blue curve and the equilibrium $(x_\star,F_1(x_\star),0)$ is marked with a blue $\ast$-mark. It is globally stable in the $xy$-plane but unstable in the $z$-direction. The remaining equilibria $(0,0,0)$ and $(1,0,0)$ in the predator-prey plane are marked with red $\ast$-marks. In the $xz$-plane, the isoclines $F_2$ and $G$ are marked with dotted blue curves and lines, respectively, and their intersection at $(x_+,0,F_2(x_+))$ is marked with a red $\ast$-mark. The equilibrium at $(0,0,c)$ is marked with a red $\circ$-mark. The limit cycle in the $xz$-plane is marked with a red curve. It approaches the point $(0,0,c)$ from its stable manifold along the $z$-axis and leaves it along its unstable manifold in the $xz$-plane and spends a considerable amount of time in the vicinity of this saddle. The projection of the limit cycle onto the $xy$-plane is marked with a cyan curve and the projection onto the $yz$-plane is indicated in green. There is no need to visualize the projection onto the $xz$-plane, but if this projection would not agree with the limit cycle itself, it would appear in magenta. The idea of selecting the value $x_\star=.26$ very close to the bifurcation value $x_\star=.2598$ was to increase the possibility for observing more complex oscillations involving all three species in this region. We did not find any oscillatory behavior of this type at this stage of our study.

We now turn over to the region $x_\star<.2598$. In this case, the herbivore-autotroph coexistence is always described by a unique limit cycle. Therefore, we do not expect qualitatively destabilizing behavior after a mixotrophic invasion in the sense that we previously found in region {\bf{(j)}}. The area corresponding to an unstable coexistence equilibrium is marked with yellow or black dots in regions {\bf{(g)}}, {\bf{(k)}}, and {\bf{(u)}}, green dots in regions {\bf{(d)}} and {\bf{(h)}}, and cyan dots in regions {\bf{(p)}} and {\bf{(t)}}. If $a_2>4$, $(0,0,c)$ is always stable (areas {\bf{(a)}}, {\bf{(c)}}, {\bf{(d)}}, and {\bf{(e)}}). No coexistence equilibrium nor competition equilibria exist in region {\bf{(a)}}. In regions {\bf{(c)}} and {\bf{(e)}}, two competition equilibria exist but no coexistence equilibrium and in region {\bf{(d)}}, two competition equilibria coexist together with an unstable coexistence equilibrium. We selected the parameter values $x_\ast=.05$ and $a_2=4.5$ to illustrate the situation in region {\bf{(d)}} (marked by a black * in Figure \ref{bifurcation_diagram}). The result is depicted in Figure \ref{Simubild22}(b). All positive solutions seem to converge towards the mixotroph carrying capacity $(0,0,c)$ that is denoted by a red $\circ$-mark. The washout equilibrium $(0,0,0)$, the carrying capacity equilibrium $(1,0,0)$ and the two competition equilibria are denoted by red $\ast$-marks. The autotroph-herbivore equilibrium $(x_\star,F_1(x_\star),0)$ is denoted by a blue $\ast$-mark. The coexistence equilibrium $(x_\star,y_\star,z_\star)$ is denoted by a black $\ast$-mark together with its projections in the autotroph-herbivore plane (cyan), competition-plane (magenta), and in the $yz$-plane (green). Finally, the autotroph-herbivore cycle in the predator-prey plane is plotted in blue and is unstable with respect to mixotropic invasion (cf. Table \ref{bifurcation_table}).

In region {\bf{(h)}}, we have again no stable equilibria so the resulting dynamics must be oscillatory. We also notice that neither $(x_+,0,F_2(x_+))$ nor $(x_\star,F_1(x_\star),0)$ are stable in the competition plane and the autotroph-herbivore plane, respectively. We selected the parameter value $x_\star=.08$ and $a_2=3.8$. The parameter value is indicated with a $\circ$-mark in region {\bf{(h)}} in our bifurcation diagram Figure \ref{bifurcation_diagram}. The resulting dynamics is depicted in Figure \ref{Simubild22}(c). We have an autotroph-herbivore cycle that is indicated in blue and is unstable with respect to mixotroph invasion. After invasion, the mixotroph outcompetes the herbivore and the resulting dynamics are the mixotroph-autotroph cycles indicated in red.

At the boundary between regions {\bf{(g)}} and {\bf{(h)}} there is a transcritical of cycles bifurcation so that the herbivore is not out-competed anymore after mixotrophic invasion in region {\bf{(g)}}. We illustrate this situation with the parameter value $x_\star=.09$ and $a_2=3.7$. The selected parameter value is indicated with a blue $\ast$-mark in Figure \ref{bifurcation_diagram}. The resulting dynamics is depicted in Figure \ref{Simubild22}(d). All the equilibria are still indicated and the unstable coexistence equilibrium is still indicated along with its projections on the various coordinate planes. The unstable competition cycle is indicated in red together with the autotroph-herbivore cycle in blue. The newly emerged coexistence cycle is indicated in black along with its projections in magenta, cyan, and green on the various coordinate planes.

Similarly, at the boundary between the regions {\bf{(g)}} and {\bf{(k)}}, a Hopf bifurcation occurs that makes the competition equilibrium $(x_+,0,F_2(x_+))$ stable in the competition plane in region {\bf{(k)}}. We selected three parameter values from this region and the first one is $x_\star=.1$ and $a_2=3.5$ (marked by a top blue * in region {\bf{(k)}} of Figure \ref{bifurcation_diagram}). The situation is depicted in Figure \ref{Simubild11}(a): The coexistence cycles that emerged from the transcritical bifucation between regions {\bf{(g)}} and {\bf{(h)}} still exists, but the cycles in the competition plane have disappeared. The projections of the coexistence limit sets on the various coordinate planes are still depicted.

The next value of interest is $x_\star=.18$, $a_2=2.8$. The parameter value is indicated with a blue $\ast$-mark in region {\bf{(k)}} in Figure \ref{bifurcation_diagram}). We have depicted the dynamical situation in Figure \ref{Simubild11}(b). We observe that the coexistence equilibrium has just lost its stability in the Hopf bifurcation predicted by Theorem \ref{complexity-stability-relation}, that all equilibria are unstable and that the system possesses coexistence cycles.

We are of course interested in whether more complex oscillations could exist and illustrate our findings with the parameter value $x_\star=.15$ and $a_2=3.0$ in Figure \ref{Simubild11}(c) (marked by a red *-mark in region {\bf{(k)}} of Figure \ref{bifurcation_diagram}). We note that the oscillations observed are a mixture of the competition cycles that emerged from the transcritical bifurcation in the competition plane and the Hopf-bifurcation occurring for the coexistence equilibrium after that the Hopf-bifurcation has occurred for the autotroph-herbivore equilibrium according to Theorem \ref{complexity-stability-relation}. The stable and unstable manifold of two equilibria are essential for these oscillations and we start the description of the oscillation from the competition equilibrium $(x_+,0,F_2(x_+))$. It is stable in the competition plane but unstable with respect to herbivore invasion. The oscillation approaches this equilibrium close to the stable manifold (the competition plane) of this point and it leaves it along the its unstable manifold. Next the mixotroph carrying capacity has the $yz$-plane as its stable manifold and an unstable manifold that intersects the two-dimensional stable manifold of the competition equilibrium.

The oscillations might come into the vicinity of the coexistence equilibrium before they are captured by the stable manifold of the mixotroph equilibrium and such an interaction is possible to observe in Figure \ref{Simubild11}(c). This interaction is, however, not essential. We now remove the coexistence equilibrium by selecting the parameter values $x_\star=.15$ and $a_2=2.3$ that is indicated by a red $\ast$-mark in region {\bf{(q)}} in Figure \ref{bifurcation_diagram}. We see that a very similar type of oscillations persist despite that the coexistence equilibrium is removed (Figure \ref{Simubild11}(d)). The two-dimensional stable manifolds of $(x_+,0,F_2(x_+))$ and $(0,0,c)$ interact with their one-dimensional unstable manifolds. We have so far no reliable indications of chaotic behavior for this system. All computations of Lyapunov exponents have so far resulted in either negative or values close to zero (from cyclic/quasiperiodic oscillations). Moreover, the possible route to chaos remains so far unclear in our simulations. The oscillations that occur without the presence of an interior fixed point in region {\bf{(q)}} and {\bf{(w)}} collapse through a transcritical of cycles bifurcation into autotroph-herbivore cycles on the boundary between the regions {\bf{(q)}} and {\bf{(w)}} and the cyan regions {\bf{(r)}} and {\bf{(x)}} that correspond to autotroph-herbivore cycles.

A similar transcritical of cycles bifurcations might occur at the boundaries between the regions {\bf{(v)}} and {\bf{(w)}}, {\bf{(t)}} and {\bf{(u)}}, and {\bf{(k)}} and {\bf{(p)}}, respectively. However, at these boundaries we have just vague numerical evidence for such transitions without theoretical support. The observed transition might be caused by numerical instabilities or can be of more complex nature than just a transcritical of cycles bifurcation. On the other hand, more theoretical evidence for transcritical bifurcations along the boundaries between {\bf{(q)}} and {\bf{(r)}} and {\bf{(w)}} and {\bf{(x)}} exists, the transcritical boundary meets e. g. the transcritical bifurcation of the coexistence equilibrium between the areas {\bf{(n)}} and {\bf{(s)}}. We select still a number of parameter values for separate study in Figure \ref{Simubild3}. We start with the parameter value $x_\star=.02$, $a_2=.6$ in region {\bf{(u)}} in Figure \ref{Simubild3}(a). Here an unstable coexistence equilibrium coexists with a coexistence cycle. Figure \ref{Simubild3}(b) demonstrated the parameter value $x_\star=.05$, $a_2=.5$ in Region {\bf{(w)}}. The coexistence cycle exists, but no coexistence equilibrium. Figure \ref{Simubild3}(c) demonstrates with the parameter value $x_\star=.2$, $a_2=.5$ in region {\bf{(x)}} how the coexistence cycle has collapsed into the herbivore-autotroph cycles in the predator-prey plane. No competition equilibrium exists in this region. Figure \ref{Simubild3}(d) demonstrates a similar transcritical collapse in region {\bf{(r)}}, where a competition equilibrium exists.

Ending up at the expectations of the title of this paper, first we have a typical qualitatively destabilizing region at region {\bf{(j)}} in Figure \ref{bifurcation_diagram} and here fixed point dynamics is replaced by limit cycle dynamics as mixotrophs invade. Second, we have a typical qualitatively stabilizing region at region {\bf{(m)}} in Figure \ref{bifurcation_diagram}. Here limit cycle dynamics is replaced by fixed point dynamics as mixotrophs invade. Third, we have typical region where mixotrophic invasion causes multiple attractors and initial value dependent behavior, see Figure \ref{multiple_attractors}(d).

\begin{figure}
\epsfxsize=173mm
\begin{picture}(350,350)(0,0)
\put(-56,-30){\epsfbox{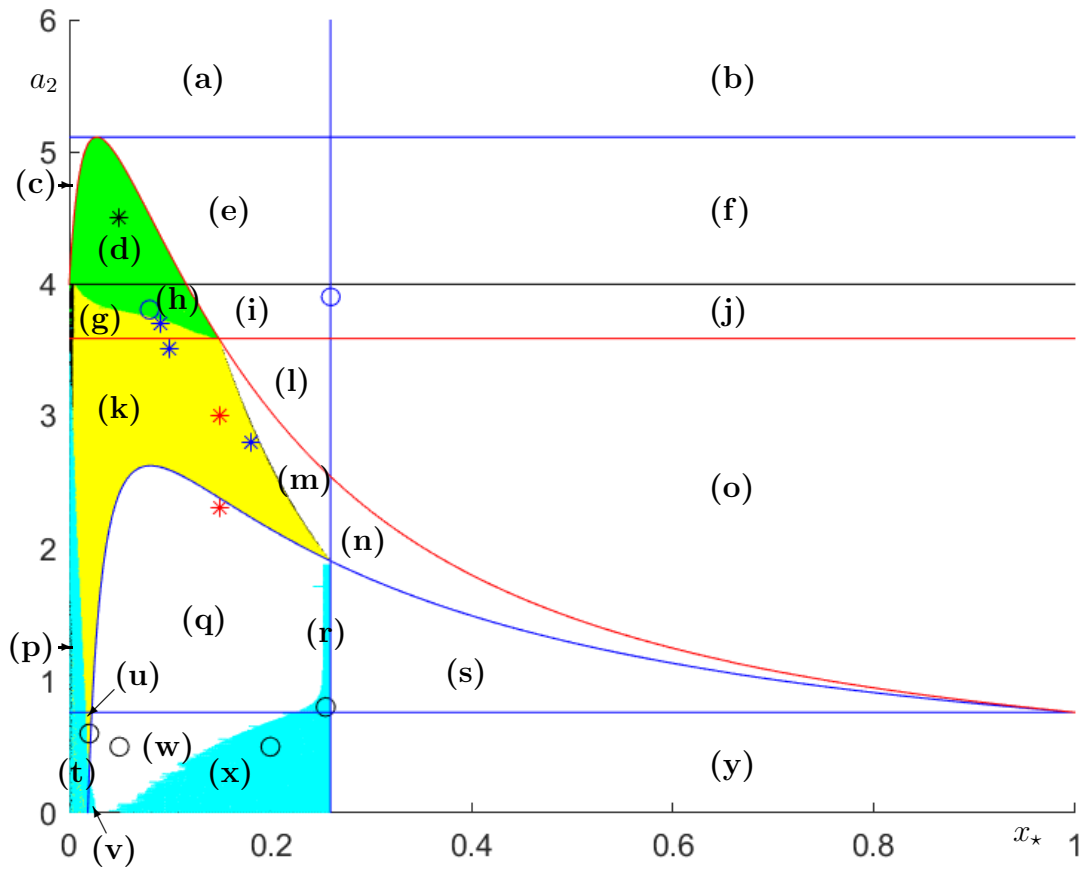}}
\put(365,0){$x_\star$}
\put(-7,285){$a_2$}
\put(-13,245){{\rm{{\bf{(c)}}}}}
\put(4,248){\vector(1,0){5}}
\put(50,285){\rm{{\bf{(a)}}}}
\put(250,285){\rm{{\bf{(b)}}}}
\put(60,235){\rm{{\bf{(e)}}}}
\put(250,235){\rm{{\bf{(f)}}}}
\put(70,197){\rm{{\bf{(i)}}}}
\put(250,197){\rm{{\bf{(j)}}}}
\put(85,170){\rm{{\bf{(l)}}}}
\put(250,130){\rm{{\bf{(o)}}}}
\put(86,133){\rm{{\bf{(m)}}}}
\put(110,110){\rm{{\bf{(n)}}}}
\put(50,80){\rm{{\bf{(q)}}}}
\put(97,75){\rm{{\bf{(r)}}}}
\put(150,60){\rm{{\bf{(s)}}}}
\put(60,22){\rm{{\bf{(x)}}}}
\put(250,25){\rm{{\bf{(y)}}}}
\put(18,220){\rm{{\bf{(d)}}}}
\put(11,194){\rm{{\bf{(g)}}}}
\put(40,201){\rm{{\bf{(h)}}}}
\put(18,160){\rm{{\bf{(k)}}}}
\put(35,32){\rm{{\bf{(w)}}}}
\put(25,59){\rm{{\bf{(u)}}}}
\put(25,57){\vector(-1,-1){10}}
\put(-15.5,70){{\rm{{\bf{(p)}}}}}
\put(4,73){\vector(1,0){5}}
\put(3,22){{\rm{{\bf{(t)}}}}}
\put(16,-7){{\rm{{\bf{(v)}}}}}
\put(22,3){\vector(-1,2){5}}
\end{picture}
\caption{The $(x_\star,a_2)$ parameter plane in the case $c=.2$, $k=.95$, $a_1=8.5$, $b_1=50$, $b_2=55$. Parameter values selected for special study are marked with either $\circ$- or $\ast$-marks. Corresponding dynamical behavior is summarized in Table \protect\ref{bifurcation_table}.}
\label{bifurcation_diagram}
\end{figure}

\begin{table}
\tiny{
\begin{tabular}{|l|l|l|l|}\hline
Region    & Predator-prey plane                 & Competition plane & Global dynamics \\ \hline
{\bf{(a)}} & Unique limit cycle unstable   & $(0,0,c)$ stable & $(0,0,c)$ stable \\
    & w.r.t. mixotrophic invasion    &                &  No coexistence equilibria              \\ \hline
{\bf{(b)}} & Equilibrium unstable w.r.t.  & $(0,0,c)$ stable & $(0,0,c)$ stable \\
    & mixotrophic invasion &                & No coexistence equilibria              \\ \hline
{\bf{(c)}} & Unique limit cycle unstable   & alternative competition eq. & $(0,0,c)$ stable \\
    & w.r.t. mixotrophic invasion    & exist, all of them unstable & No coexistence equilibria            \\
    &                                & w.r.t herbivore invasion  &                                        \\\hline
{\bf{(d)}} & Unique limit cycle unstable   & alternative competition & $(0,0,c)$ stable \\
    & w.r.t. mixotrophic invasion    & equilibria exist & unstable coexistence equilibrium exists                \\
    &                                & $(x_+,0,F_2(x_+))$ unstable w.r.t                 &                                                     \\
    &                                & herbivore invasion                                 &                                  \\ \hline
{\bf{(e)}} & Unique limit cycle unstable   & alternative competition  & $(0,0,c)$ stable \\
    & w.r.t. mixotrophic invasion    & equilibria exist, all stable & No coexistence equilibria               \\
    &                                & w.r.t. herbivore invasion    &                    \\ \hline
{\bf{(f)}} & Equilibrium unstable w.r.t.  & alternative competition & $(0,0,c)$ stable \\
    & mixotrophic invasion    & equilibria exist all stable &  No coexistence equilibria             \\
    &                         & w.r.t. herbivore invasion &                            \\ \hline
{\bf{(g)}} & Unique limit cycle unstable   & Limit cycle unstable w.r.t. & Oscillatory coexistence dynamics\\
    & w.r.t. mixotrophic invasion    & herbivore invasion & unstable coexistence equilibrium exists              \\ \hline
{\bf{(h)}} & Unique limit cycle unstable   & stable competition  & stable competition limit cycle \\
    & w.r.t. mixotrophic invasion    & limit cycle & unstable coexistence equilibrium exists               \\ \hline
{\bf{(i)}} & Unique limit cycle unstable  & Stable competition  & Stable competition limit cycle \\
    & w.r.t. mixotrophic invasion    & limit cycle &  No coexistence equilibria  \\ \hline
{\bf{(j)}} & Equilibrium unstable w.r.t.  & Stable competition  & Stable competition limit cycle \\
    & mixotrophic invasion    & limit cycle &  No coexistence equilibria  \\ \hline
{\bf{(k)}} & Unique limit cycle unstable  & Equilibrium unstable w.r.t. & Oscillatory coexistence dynamics\\
    & w.r.t. mixotrophic invasion    & herbivore invasion & unstable coexistence equilibrium exists               \\ \hline
{\bf{(l)}} & Unique limit cycle unstable   & Stable competition  & Stable competition equilibrium \\
    & w.r.t. mixotrophic invasion    & equilibrium & No coexistence equilibria  \\ \hline
{\bf{(m)}} & Unique limit cycle unstable   & Equilibrium unstable w.r.t. & Stable coexistence equilibrium \\
    & w.r.t. mixotrophic invasion    & herbivore invasion &    \\ \hline
{\bf{(n)}} & Equilibrium unstable w.r.t.  & Equilibrium unstable w.r.t. & Stable coexistence equilibrium \\
    & mixotrophic invasion    & herbivore invasion &  \\ \hline
{\bf{(o)}} & Equilibrium unstable w.r.t.  & Stable competition  & Stable competition equilibrium \\
    & mixotrophic invasion    & equilibrium & No coexistence equilibria   \\ \hline
{\bf{(p)}} & Stable unique  & Equilibrium unstable w.r.t. & Predator-prey cycle stable?\\
    & limit cycle?    & herbivore invasion & unstable coexistence equilibrium exists               \\ \hline
{\bf{(q)}} & Unique limit cycle unstable   & Equilibrium unstable w.r.t. & No coexistence equilibria \\
    & w.r.t. mixotrophic invasion    & herbivore invasion & Oscillatory coexistence dynamics   \\ \hline
{\bf{(r)}} & Stable unique         & Unstable competition        & Unique Predator-prey limit cycle stable \\
    & limit cycle    & equilibrium, $(0,0,c)$ unstable &  No coexistence equilibria \\ \hline
{\bf{(s)}} & Stable  & Equilibrium unstable w.r.t. & Predator-prey equilibrium stable \\
    & Equilibrium    & herbivore invasion & No coexistence equilibria    \\ \hline
{\bf{(t)}} & Stable unique       & No competition equilibria         & Predator-prey cycle stable? \\
    & limit cycle?    & $(0,0,c)$ unstable & unstable coexistence equilibrium exists     \\ \hline
{\bf{(u)}} & Unique limit cycle unstable       & No competition equilibria         & Oscillatory coexistence dynamics \\
    & w.r.t. mixotrophic invasion    & $(0,0,c)$ unstable & unstable coexistence equilibrium exists     \\ \hline
{\bf{(v)}} & Stable unique       & No competition equilibria         & Predator-prey cycle stable? \\
    & limit cycle?    & $(0,0,c)$ unstable & No coexistence equilibrium     \\ \hline
{\bf{(w)}} & Unique limit cycle unstable         & No competition equilibria         & Oscillatory coexistence dynamics\\
    & w.r.t  mixotrophic invasion    & $(0,0,c)$ unstable & No coexistence equilibria    \\ \hline
{\bf{(x)}} & Stable unique        & No competition equilibria        & Predator-prey limit cycle stable\\
    & limit cycle    & $(0,0,c)$ unstable & No coexistence equilibria    \\ \hline
{\bf{(y)}} & Stable         & No competition equilibria         & Predator-prey equilibrium stable \\
    & Equilibrum    & $(0,0,c)$ unstable & No coexistence equilibria   \\ \hline
\end{tabular}}
\caption{Major topological properties of (\protect\ref{chemostat_mixo_sat_scaled_altII}) related to the different regions in Figure \ref{bifurcation_diagram}. The description of what actually happens at the areas {\bf{(p)}}, {\bf{(t)}}, and {\bf{(v)}} is incomplete and we have so far no theoretical evidence confirming the numerical observations reported here that might to be subject to instabilities and round-off errors in our numerical implementation of (\protect\ref{chemostat_mixo_sat_scaled_altII}).}
\label{bifurcation_table}
\end{table}

\begin{figure}
\epsfxsize=138mm
\begin{picture}(350,350)(0,0)
\put(0,0){\epsfbox{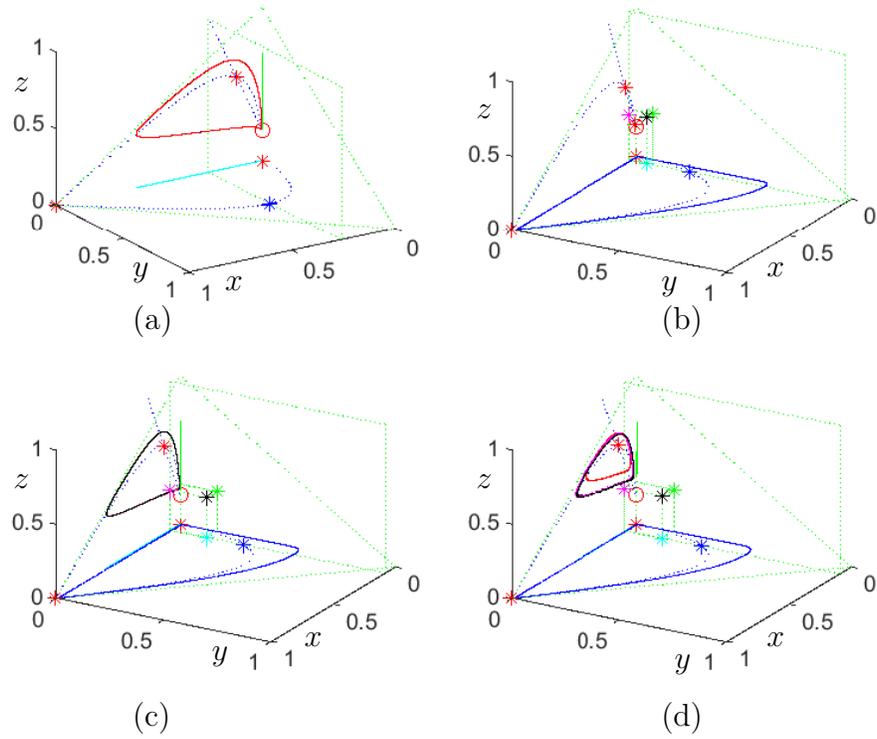}}
\put(80,150){(a)}
\put(280,150){(b)}
\put(80,0){(c)}
\put(280,0){(d)}
\put(35,240){$z$}
\put(35,90){$z$}
\put(210,230){$z$}
\put(210,90){$z$}
\put(80,170){$y$}
\put(280,165){$y$}
\put(110,25){$y$}
\put(285,22){$y$}
\put(115,165){$x$}
\put(320,170){$x$}
\put(145,30){$x$}
\put(320,30){$x$}
\end{picture}
\caption{Simulations in the case $c=.2$, $k=.95$, $a_1=8.5$, $b_1=50$, $b_2=55$. (a) Region {\bf{(j)}}: $x_\star=.26$, $a_2=3.9$. (b) Region {\bf{(d)}}: $x_\star=.05$, $a_2=4.5$. (c) Region {\bf{(h)}}: $x_\star=.08$, $a_2=3.8$. (d) Region {\bf{(g)}}: $x_\star=.09$, $a_2=3.7$.}
\label{Simubild22}
\end{figure}

\begin{figure}
\epsfxsize=138mm
\begin{picture}(350,350)(0,0)
\put(0,0){\epsfbox{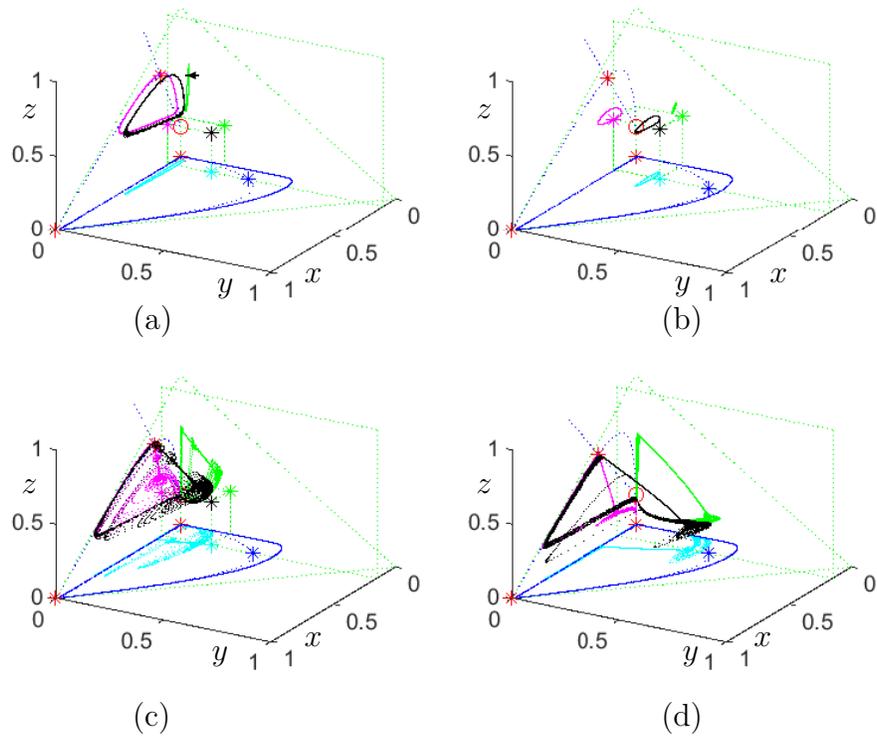}}
\put(80,150){(a)}
\put(280,150){(b)}
\put(80,0){(c)}
\put(280,0){(d)}
\put(38,230){$z$}
\put(38,88){$z$}
\put(210,230){$z$}
\put(210,88){$z$}
\put(112,164){$y$}
\put(285,165){$y$}
\put(110,25){$y$}
\put(288,25){$y$}
\put(145,168){$x$}
\put(321,168){$x$}
\put(145,30){$x$}
\put(315,30){$x$}
\put(105,246){\vector(-1,0){5}}
\end{picture}
\caption{Simulations in the case $c=.2$, $k=.95$, $a_1=8.5$, $b_1=50$, $b_2=55$. (a) Region {\bf{(k)}}: $x_\star=.1$, $a_2=3.5$. (b) Region {\bf{(k)}}: $x_\star=.18$, $a_2=2.8$. (c) Region {\bf{(k)}}: $x_\star=.15$, $a_2=3.0$. (d) Region {\bf{(q)}}: $x_\star=.15$, $a_2=2.3$.}
\label{Simubild11}
\end{figure}

\begin{figure}
\epsfxsize=138mm
\begin{picture}(350,350)(0,0)
\put(0,0){\epsfbox{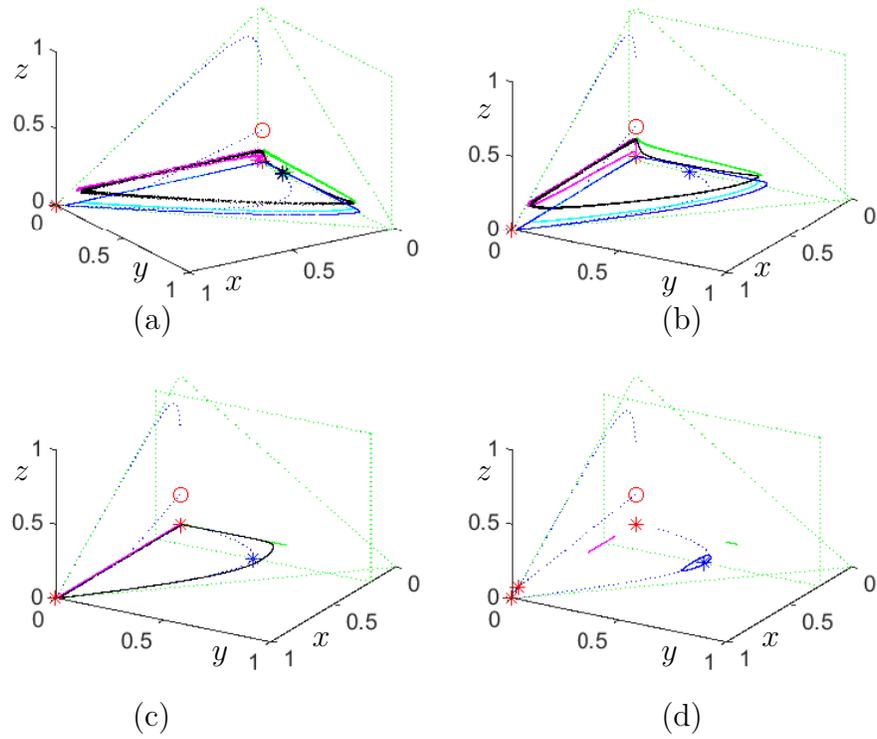}}
\put(80,150){(a)}
\put(280,150){(b)}
\put(80,0){(c)}
\put(280,0){(d)}
\put(35,245){$z$}
\put(35,93){$z$}
\put(210,230){$z$}
\put(210,93){$z$}
\put(80,170){$y$}
\put(280,165){$y$}
\put(110,25){$y$}
\put(280,25){$y$}
\put(115,165){$x$}
\put(315,170){$x$}
\put(148,30){$x$}
\put(318,30){$x$}
\end{picture}
\caption{Simulations in the case $c=.2$, $k=.95$, $a_1=8.5$, $b_1=50$, $b_2=55$. (a) Region {\bf{(u)}}: $x_\star=.02$, $a_2=.6$. (b) Region {\bf{(w)}}: $x_\star=.05$, $a_2=.5$. (c) Region {\bf{(x)}}: $x_\star=.2$, $a_2=.5$. (d) Region {\bf{(r)}}: $x_\star=.255$, $a_2=.8$.}
\label{Simubild3}
\end{figure}

\section{Discussion}
\label{dicussion}
In the present study, we have examined the role of invading mixotrophs on the stability of an autotroph-herbivore system. Specifically, we have chosen a spectrum of mixotrophic organisms by varying their competitive abilities. Mixotrophs are assumed as weak competitors compared to the pure autotrophs and the pure herbivores (Litchman et al. (2007)\nocite{Litchman.EcolLett:10}; Zubkov and Tarran (2008)\nocite{Zubkov.Nature:455}). We are able to formulate such a condition precisely for the mixotrophs in relation to the autotrophs, but in relation to the herbivores we end up with a too strong global condition and a too weak local condition. However, we did not encounter any obvious problem with our analysis as we used the local condition. We have analyzed the system with a limiting case of the chemostat instead of doing logistic approximations since the limiting case still preserves the complete information of the system exposed to mixotrophic invasion (Lindstr\"{o}m and Cheng (2015)\nocite{lindstr_cheng}) similar to the logistic approximations (Kuang and Freedman (1988)\nocite{kuangfyra}).

Two main subsystems of our system, a competition system (autotroph-mixotroph) and a predator-prey system (autotroph-herbivore), describe the dynamical properties in different invariant coordinate planes. The predator-prey system has a $\omega$-limit set that consists of either a unique globally stable equilibrium or a unique limit cycle. As the cycles grow larger, the system becomes increasingly sensitive to perturbation, cf. Rosenzweig (1971)\nocite{Rosenzweig.Science:171}. The global analysis is not that complete for the competition system. It has at most five equilibria and both situations where all existing equilibria are unstable and two of them remain locally stable exist. As a result, this subsystem possesses both initial value dependent behavior and limit cycles.

Our analysis also reveals the existence of a possible additional unstable equilibrium in the complete three-species model. This leads to a situation without any stable equilibrium and opens the possibility for complicated oscillations. However, we were unable to find any evidence for chaos at this stage but found that the oscillations observed typically spend a lot of time close to some of the saddle-type equilibria. Checking the existence of chaos needs a lot of effort in computing Lyapunov exponents in a reliable way and is thus out of the scope of this paper (Guckenheimer and Holmes (1983)\nocite{guck}, and Wolf et al. (1985)\nocite{wolf}).

We have analyzed the system by using two major bifucation parameters, one that controls the competitive ability of the mixotroph compared to the autotroph for limiting resources ($a_2$-the mixotrophic link parameter), and the other one that controls the autotroph-herbivore interaction ($x_\star$-the equilibrium of the autotroph when interacting with herbivores). For small $x_\star$, the coexistence between autotrophs and herbivores is described by unique limit cycles, whereas as $x_\star$ becomes large, a globally stable fixed point describes the situation. On the other hand, when $a_2$ remains small, autotroph appears as a better competitor for nutrient than the mixotroph and outcompetes the mixotroph. As $a_2$ increases, first the mixotroph and the autotroph coexist in a stable equilibrium, and then a Hopf-bifurcation occurs resulting in autotroph-mixotroph cycles. Further increase in $a_2$ results in the existence of multiple competition equilibria together with a stable equilibrium at the mixotrophic carrying capacity. For very large values of $a_2$, mixotroph outcompetes the autotroph and the mixotrophic carrying capacity equilibrium (0, 0, c) becomes stable.

The most interesting result of this study is that invading mixotrophs may both stabilize and destabilize existing autotroph-herbivore dynamics. We have provided substantial analytical and numerical evidence for our claim. Globally stable two-species interaction may turn into competition cycles after the grazer has been outcompeted and the cyclic dynamics might be converted into either fixed points or oscillatory dynamics that may involve all, just two, or only one species. The appearance of multiple attractors is not only something that can occur for the autotroph-mixotroph relation, but a behavior that persists even in the complete three-species system. Note that, in a similar type of system but with relatively simpler predation term, Jost et al. (2004)\nocite{Jost.TPB:66} found only stabilizing role of mixotrophy. Other modelling studies support the claim of stabilizing nature of mixotrophy (Hammer and Pitchford, 2005\nocite{Hammer.ICESMarine:62}; Mitra et al. 2014\nocite{Mitra.PoO:129}). There is also evidence that the invasion of mixotrophs destabilizes the system by importing oscillations (Crane and Grover 2010). However, both stabilizing and destabilizing roles of invading mixotrophs in a single model has been found for the first time in the present study.

The model presented here was used to show only the fundamental mechanisms and must not be confused with quantitative simulations. The main purpose was to investigate the role of mixotrophs in a simple predator-prey system in order to understand the conditions under which invasion of mixotrophs is possible and how it affects the stability of the system and leads to more complicated dynamics. The occurrence of multi-stability and both stabilization and destabilization effects associated with the invasive mixotrophs indicate that nutrient limited lower latitude areas and higher latitude summer conditions which are ideal for mixotrophs to dominate, the community composition might be very sensitive to any kind of perturbations, i.e. a small change in the environmental condition can result into a huge change in the community composition and finally ecosystem functions. However, to quantify the real effects we need to run our system in a real physical setup using more realistic parameter values.

\paragraph{Acknowledgements} Some of the questions treated in this paper emanated from a discussion between Prof. em. Edna Gran{\'{e}}li and Torsten Lindstr\"{o}m that was connected to a joint research project proposal 2004, but national Swedish funding was never admitted. Consequently, it took some time to arrange possibilities for completing this study. Torsten Lindstr\"{o}m thanks the Department of Mathematics at University of Helsinki for a stimulating discussion after a seminar lecture November 8, 2017, based on an earlier version of this paper. Subhendu Chakraborty was supported by the H. C. \O rsted COFUND postdoc fellowship.

\bibliographystyle{abbrv}
\bibliography{artiklar,biologi,dynamic}

\end{document}